\documentclass[onetabnum]{siamart171218}

\usepackage{lipsum}
\usepackage{amsfonts}
\usepackage{graphicx}
\usepackage{epstopdf}
\usepackage{algorithmic}
\usepackage{enumerate}
\usepackage[toc]{appendix}
\usepackage{multirow}
\ifpdf
  \DeclareGraphicsExtensions{.eps,.pdf,.png,.jpg}
\else
  \DeclareGraphicsExtensions{.eps}
\fi

\newtheorem{remark}[theorem]{Remark}
\def\RR{\mathbb{R}}

\headers{Order reduction methods for Differential Riccati equations}{G. Kirsten and V. Simoncini}

\title{Order reduction methods for solving large-scale differential matrix Riccati equations \thanks{
This version dated December 20, 2019} }

\author{Gerhard Kirsten\thanks{Dipartimento di Matematica,
Universit\`a di Bologna, Piazza di Porta S. Donato, 5, I-40127 Bologna, Italy,
({\tt gerhard.kirsten2@unibo.it})}
\and
Valeria Simoncini\thanks{Dipartimento di Matematica and ${\rm AM}^2$,
Universit\`a di Bologna, Piazza di Porta S. Donato, 5, I-40127 Bologna, Italy;
and IMATI-CNR, Pavia Italy ({\tt valeria.simoncini@unibo.it}).}}

\usepackage{amsopn}

\begin{document}

\maketitle

\begin{abstract}
We consider the numerical solution of large-scale symmetric differential matrix Riccati 
equations. Under certain hypotheses on the data, reduced order methods have recently 
arisen as a promising class of solution strategies, by forming low-rank approximations to 
the sought after solution at selected timesteps. 
We show that great computational and memory savings are obtained by a reduction
process onto rational Krylov subspaces, as opposed to current approaches. 
By specifically addressing the solution of the reduced differential equation and
reliable stopping criteria, we are able to 
obtain accurate final approximations at low memory and computational requirements.
This is obtained by employing a two-phase strategy that separately enhances the accuracy of
the algebraic approximation and the time integration. The new method allows us to numerically solve
much larger problems than  in the current literature.
 Numerical experiments on benchmark problems illustrate the effectiveness of the procedure with respect
to existing solvers. 
\end{abstract}



\section{Introduction}
We consider the solution of the continuous-time differential matrix Riccati equation (DRE in short) of the form
\begin{equation}
\dot{X}(t) = A^{T}X(t) + X(t)A - X(t)BB^{T}X(t) + C^{T}C, \quad X(0) = X_{0}, 
\label{dre}
\end{equation}
in the unknown matrix $X(t)\in \mathbb{R}^{n \times n}$, where $X_{0} = Z Z^{T}$ and $t \in [0,t_{f}]$. 
Here, $A \in \mathbb{R}^{n \times n}$, $B \in \mathbb{R}^{n \times s}$, $C \in \mathbb{R}^{p \times n}$ 
and $Z \in \mathbb{R}^{n \times q}$ are time invariant, and $s, p, q \ll n$. 
The matrix $A$ is assumed to be large, sparse and nonsingular, whereas $B$, $C$ and $Z$ have full rank. 
In particular, we consider low-rank DREs, where both matrices $C^TC$ and $X_0$ have very low rank compared
to $n$.
Even though the matrix $A$ is sparse, the solution $X(t)$ is typically 
dense and impossible to store when $n$ is large. Under the considered hypotheses, numerical evidence seems to
indicate that  $X(t)$ usually has rapidly decaying singular values, hence a low-rank 
approximation to $X(t)$ may be considered, see e.g., \cite{Stillfjord2018LR}. 
For completeness, we also refer the reader to \cite{grasedyck2004, grasedyck2004.2} for 
results on the existence of low-rank solutions for the algebraic Sylvester and Lyapunov (linear)
equations.

The DRE plays a fundamental role in optimal control theory, filter design theory, model reduction problems, as well as in
differential games \cite{AbouKandil2003,BCOW.17,Blanes2015,CORLESS2003,Reid72}.  
{Equations of the form \cref{dre} are crucial in the numerical treatment of 
  the linear quadratic regulator (LQR) problem 
\cite{AbouKandil2003, CORLESS2003, Kwakernaak1972}: given the state equation
\begin{equation} \label{state}
\dot{x}(t) = Ax(t) + Bu(t),\quad y(t)=C x(t), \qquad x(0) = x_{0}
\end{equation}
consider the finite horizon case, where the finite time cost integral has the form
\begin{equation}
J(u) = x(t_{f})^{T}P_{f}x(t_{f}) + \int_{0}^{t_{f}} \left( x(t)^{T}C^{T}Cx(t) + u(t)^{T}u(t) \right) dt.
\label{cost}
\end{equation} 
The matrix $P_{f}$ is assumed to be symmetric and nonnegative definite. 
Assuming that the pair $(A,B)$ is stabilizable and the pair $(C,A)$ is detectable, 
the optimal input $\tilde{u}(t)$, minimizing \cref{cost}, can be determined as
$\tilde{u}(t) = -B^{T}P(t) \tilde{x}(t)$,
and the optimal trajectory is subject to $\dot{\tilde{x}} = (A - BB^{T}P(t))x(t)$. The matrix
$P(t)$ is the solution to the DRE  
\begin{equation}
\dot{P}(t) = A^{T}P(t) + P(t)A - P(t)BB^{T}P(t) + C^{T}C, \quad P(t_{f}) = P_{f}.
\label{drefinals}
\end{equation}
Using a common practice, we can transform \cref{drefinals} into the 
initial value problem \cref{dre} via the change of variables $X(t_{f} - t) = P(t)$.

Under certain assumptions, the exact solution of \cref{dre} can be expressed in integral form as (see e.g., \cite[Theorem 8]{Kucera1973})
\begin{equation} \label{exactsolution}
X(t) = e^{tA^{T}}ZZ^{T}e^{tA} + \int_{0}^{t} e^{(t-s)A^{T}}\left (C^{T}C- X(s)BB^{T}X(s)\right ) e^{(t-s)A} ds  ,
\end{equation} 
so that when $t \rightarrow \infty$ the DRE reaches a steady state solution satisfying the algebraic Riccati equation (ARE)
\begin{equation}
\label{are}
0 = A^{T}X_{\infty} + X_{\infty}A - X_{\infty}BB^{T}X_{\infty} + C^{T}C .
\end{equation}
In the framework of differential equations, the DRE is characterized by
 both fast and slow varying modes, hence it is classified as a stiff ordinary differential equation (ODE). 
The stiffness and the nonlinearity of the DRE are responsible for the difficulties in
its numerical solution even on a small scale ($n < 10^{3}$). 
Several stiff integrators have been investigated, including the matrix generalizations of 
implicit ODE solvers \cite{Dieci1992, Choi1990a}, linearization methods \cite{Davison1973} 
and more recently matrix versions of 
splitting methods \cite{Mena2018, Stillfjord2015, Stillfjord2018}. These methods are
feasible on a small scale but fail to be efficient when $n$ is large. 
In \cite{Saak2016}, iterative methods are implemented within the 
matrix generalization of standard implicit methods allowing for the computation of an approximate 
solution to the DRE when $n \gg 10^{3}$. These algorithms 
require the solution of a large {\it algebraic} Riccati equation at each timestep, which 
again raises big concerns as of storage and computational efforts.

A promising idea is to rely on a model order reduction strategy typically used
in linear and nonlinear dynamical systems. In this setting, the 
original system is replaced with
\begin{equation} \label{state_red}
\dot{\widehat x}(t) = A_m \widehat x(t) + B_m u(t),\quad y(t)=C_m \widehat x(t), \qquad \widehat x(0) = \widehat x_{0}
\end{equation}
where $A_m, B_m$ and $C_m$ are projections and restrictions of the original matrices
onto a subspace of small dimension. The differential Riccati equation associated
with this reduced order problem is solved, yielding an optimal corresponding cost function.
This strategy allows for a natural low-rank approximation to the sought after DRE solution $X(t)$,
obtained by interpolating the reduced order solution at selected time instances.
One main feature is that a single space is used for all time snapshots, so that
the approximate solutions can be kept in factored form with few memory allocations.
We refer the reader to, e.g., \cite{Antoulas2005} for a general presentation of
algebraic reduction methods for linear dynamical systems, and
to \cite{Simoncini2016} for a detailed discussion motivating the reduction approach in the
context of the algebraic Riccati equation.

A key ingredient for the success of the reduction methodology
is the choice of the approximation space onto which the algebraic reduction is performed; 
\cite{Antoulas2005} presents a comprehensive description of various space selections
in the dynamical system setting.
Following strategies already successfully adopted for the algebraic Riccati equations, 
the authors of \cite{Koskela2018} and \cite{Guldogan2017} have independently used
polynomial and extended Krylov subspaces as approximation space, respectively, in the
differential setting.
A major characteristic of these spaces is that their dimension can be expanded iteratively, so that
if the determined approximate solution is not sufficiently accurate, the Krylov space can be
enlarged and the process continued.
Several questions remain open in the methods proposed in \cite{Koskela2018},\cite{Guldogan2017}.
On the one hand, it is well known that polynomial Krylov subspaces require a very large dimension
to satisfactorily solve real application problems, thus destroying the reduction advantages.
On the other hand, the multiple timestepping proposed in the method in \cite{Guldogan2017} 
only provides an accurate approximation at $t = t_{f}$, except 
when $X_0=0$. For $X_0=ZZ^T \ne 0$ of low rank, memory requirements of the extended method grow significantly.
These problems can be satisfactorily solved by using a general {\it rational} Krylov subspace, which
is shown in various applications
to be able to supply good spectral information on the involved matrices with much smaller dimensions
than the polynomial and extended versions. Such gain has been experimentally reported in
the literature in the solution of the {\it algebraic} Riccati equation.
We show that great computational and memory savings can be obtained when projecting 
onto the fully rational Krylov subspace, and that with an appropriate implementation the extended 
Krylov subspace may also be competitive with certain data.

A related issue that has somehow been overlooked in the available literature is 
the expected final accuracy and thus the
stopping criterion. Time dependence of the DRE makes the reduced problem trickier to handle than
in the purely algebraic case; in particular, two intertwined issues arise: i) The accuracy
of the approximate solution may vary considerably within the time interval $[0,t_f]$;
ii) Throughout the reduction process
the reduced ODE cannot be solved with high accuracy and, quite the opposite, low-order methods
should be used to make the overall cost feasible.
We analyze these difficulties in detail, and by exploiting the inherent structure of
the reduced order model, we derive a two-phase strategy that first focuses on the reduction
and then on the integration, in a way that is efficient for memory and CPU time usage,
but also in terms of final expected accuracy.

We also discuss several algebraic properties of the approximate solution and its relation
both with the solution $X(t)$ for $t\in [0,t_f]$, and with the steady state solution $X_\infty$.
These results continue a matrix analysis started in \cite{Koskela2018}, where positivity and
monotonicity properties of the approximate solution obtained by certain reduction methods are
explored.

{The paper is organized as follows. In \cref{sec:projection} we introduce reduction methods
and discuss the use of Krylov subspace based strategies. 
Matrix-oriented BDF methods are recalled for the solution of the projected 
problem in \cref{sec:internal}. In \cref{sec:krylovdre} we devise a stopping criterion for 
the order reduction methods and illustrate its key role in the implementation.
Section~\ref{sec:stab} is devoted to the analysis of matrix properties of the solution, as 
well as the reduced model, from a control theory perspective. Several 
numerical experiments are reported in \cref{sec:main}, where the new methods are also
compared with state-of-the-art procedures. Our conclusions are discussed in \cref{sec:conclusion}. 
Finally, in \cref{matral} and \cref{appb} we review some properties of the extended and rational Krylov subspaces.}

{\it Notation and definitions.} Throughout the paper, the matrix $I_{n}$ will denote the $n \times n$ identity matrix. In terms of norms, $\|\cdot\|$ refers to any induced matrix norm, where in particular the Frobenius norm is denoted by $\|\cdot\|_{F}$. A matrix $A$ is stable (sometimes also called Hurwitz)} if all its {\em eigenvalues} are contained in the left half open complex plane. A linear dynamical system,
$
\dot{x} = Ax,
$
is called dissipative if the real matrix $A$ has its {\em field of values} contained in the left half open complex plane.

All reported experiments were performed using MATLAB 9.4 (R2018b) 
(\cite{matlab2013}) on a MacBook Pro with 8-GB memory and a 2.3-GHz Intel core i5 processor.

\section{Order reduction with Krylov-based subspaces} \label{sec:projection}
In this section, we review Krylov-based order reduction methods and show how they are applied to the DRE. 
Krylov subspaces that have been explored in the past years have the form
{
\begin{eqnarray*}
\mathcal{K}_{m}(A, N) &=& \mbox{range}\left\{[N, AN, A^{2}N,\dots,A^{m-1}N]\right\} \qquad {\rm polynomial} \\
\mathcal{EK}_{m}(A, N) &=& \mathcal{K}_{m}(A, N) + \mathcal{K}_{m}(A^{-1}, A^{-1}N) \qquad {\rm extended} \\
\mathcal{RK}_{m}(A, N, {\pmb s}) &=& \mbox{range}\left\{ [N, (A-s_2 I)^{-1}N, 
\dots,\prod_{i=2}^m (A-s_i I)^{-1}N]\right\} \qquad {\rm rational}.
\end{eqnarray*}
}
where $N$ is a tall matrix associated with the given problem. In the rational subspace,
${\pmb s}=\{s_2, \ldots, s_m\}$ is a set of properly chosen
real or complex shifts, whose computation can be performed a priori or dynamically
during the generation of the subspace; we refer the reader to \cite{Simoncini2016a,Druskin2011} for
more complete descriptions.

Krylov-based projection methods (in short generically denoted as ${\cal K}_m$)
were first applied to ARE's in \cite{Jaimoukha1994} (polynomial spaces)
 and later improved in
\cite{Heyouni2009a} (extended space) and \cite{Simoncini2014} (rational spaces).
The two rational spaces prove to
be far superior to the polynomial Krylov space in most reduction strategies where they are applied in the
literature, 
as long as solving linear systems at each iteration is feasible.
 The differential 
Riccati equation has been attacked in \cite{Guldogan2017} with the extended space,
and in \cite{Koskela2018} with the polynomial space; here we close the gap, as far as
Krylov subspaces are concerned. In addition, we address several implementation issues
to make the final method computationally reliable and, to the best of our knowledge,
a great competitor among the available methods for large-scale DRE problems.

While for the algebraic Riccati equation $N=C^T$, in the differential context
the starting matrix for generating these spaces is given by $N=[C^T,Z]$, where $X_0=ZZ^T$.
Both matrices $C$ and $Z$ play a crucial role in the closed-form DRE solution matrix and are thus included
to generate the projection space.
The idea of reduction methods is to 
first project the large DRE onto the smaller subspace $\mathcal{K}_{m}$, 
then solve the projected equation, and finally expand the solution back to the original space. 

Let the columns of $\mathcal{V}_{m} \in \mathbb{R}^{n \times d}$ span the considered Krylov subspace. Then the following
Arnoldi-type relation holds,
\begin{equation} \label{genericrelation}
A^{T}\mathcal{V}_{m}  = \mathcal{V}_{m}\mathcal{T}_m^{T} +  \nu_{m+1}\tau_{m}^{T},
\end{equation}
where the actual values of  $\nu_{m+1}\in\RR^n$ and $\tau_{m}^{T}$ depend on the chosen subspace. Moreover,
setting ${\cal V}_{m+1}=[\mathcal{V}_{m}, \nu_{m+1}]$ we have
that $\mathcal{K}_{m+1} = {\rm range}({\cal V}_{m+1})$, which 
shows that Krylov subspaces are nested, that is $\mathcal{K}_{m} \subseteq \mathcal{K}_{m+1}$,
 resulting in a dimension increase after each iteration. Matrix relations leading to \cref{genericrelation}
for the extended and rational Krylov subspaces are recalled in \cref{matral}. 

Assume that ${\cal V}_m$ has orthonormal columns.
Following similar reduction methods in the dynamical system contexts, see, e.g., 
\cite{Antoulas2005}, the reduction process consists of first projecting and restricting the original data
onto the approximation space as
\begin{equation*}
\mathcal{T}_{m} = \mathcal{V}_{m}^{T}A\mathcal{V}_{m}, 
\quad B_{m} = \mathcal{V}_{m}^{T}B, \quad Z_{m} = \mathcal{V}_{m}^{T}Z \quad \mbox{and} \quad C_{m} = C\mathcal{V}_{m}.
\end{equation*}
Then the following low order differential Riccati equation needs to be solved,
\begin{equation} \label{projdre}
\begin{split}
\dot{Y}_{m}(t) &= \mathcal{T}_{m}^{T}Y_{m}(t) + Y_{m}(t)\mathcal{T}_{m} - Y_{m}(t)B_{m}B_{m}^{T}Y_{m}(t) + C_{m}^{T}C_{m} \\
Y_{m}(0) &= Z_{m}Z_{m}^{T},
\end{split}
\end{equation}
for $t\in [0, t_f]$. This low-dimensional DRE admits
a unique solution for $t_f < \infty$~, see e.g., \cite{Kucera1973}. 
Restrictions on the data to allow for positive, stabilizing solutions are
 discussed in more detail in section~\ref{sec:algstab}.
An approximation to the sought after solution is then written as
\begin{equation}
\label{projection}
X_{m}(t) = \mathcal{V}_{m}Y_{m}(t)\mathcal{V}_{m}^{T} \approx X(t), \qquad t\in [0, t_f].
\end{equation}
We stress that $X_{m}(t)$ is never explicitly computed,
but always referred to via the matrix $\mathcal{V}_{m}$ and the set of matrices $Y_{m}(t)$ at
given time instances. In fact, the matrices $Y_{m}(t)$ may also be numerically low rank, so that
at the end of the whole process a further reduction can be performed by truncating the eigendecomposition
of $Y_{m}(t)$ for each $t$.

\vskip 0.1in
\begin{remark}
The approach we have derived is solely based on the order reduction of the dynamical system (\ref{state}).
Nonetheless, and with some abuse of notation, the reduced DRE could have been formally obtained by
means of a Galerkin condition on the differential equation.
For $t\in [0, t_f]$ let
\begin{equation*} 
\mathcal{R}_{m}\left(t\right) := \dot{X}_{m}(t) - A^{T}X_{m}(t) - X_{m}(t)A + X_{m}(t)BB^{T}X_{m}(t) - C^{T}C
\end{equation*}
be the residual matrix for $X_m(t) = {\cal V}_m Y_m {\cal V}_m^T$. 
The matrix $Y_{m}(t)$ is thus determined by imposing that the residual
satisfies the following {\it Galerkin} condition 
\begin{equation}
\mathcal{V}_{m}^{T}\mathcal{R}_{m}\left(t\right) \mathcal{V}_{m} = 0, \quad t\in [0, t_f],
\end{equation}
 that is, $\mathcal{R}_{m}(t) \perp \mathcal{K}_{m}$ in a matrix sense, so that the residual is 
forced to belong to a smaller and smaller subspace as $\mathcal{K}_{m}$ grows. 
Substituting $X_{m}(t) = \mathcal{V}_{m}Y_{m}(t)\mathcal{V}_{m}^{T}$ into the residual matrix, the application of 
the Galerkin condition results in the projected system 
  
{\footnotesize $$
\mathcal{V}_{m}^T\big( \mathcal{V}_m\dot{Y}_m(t)\mathcal{V}_m^T
- A^{T}\mathcal{V}_m Y_m(t)\mathcal{V}_{m}^{T} - \mathcal{V}_{m}Y_{m}(t)\mathcal{V}_m^T A
+ \mathcal{V}_m Y_m(t)\mathcal{V}_m^T BB^T\mathcal{V}_m Y_m(t)\mathcal{V}_m^T - C^T C)\mathcal{V}_{m} = 0  ,
$$}
%
$\!\!$which corresponds to (\ref{projdre}).
This is rigorous as long as ${\dot X}_m = {\cal V}_m{\dot Y}_m{\cal V}_m^T$ holds. $\,\, \square$
\end{remark}

\vskip 0.1in
It is crucial to realize that, as opposed to some available methods in the literature (such as, for 
instance, \cite{Saak2016},\cite{Stillfjord2018} and the time-invariant algorithms in \cite{Lang2017}),
the approximation space is independent of the time stepping, that is
a single space range$(\mathcal{V}_{m})$ is used for all time steps. This provides enormous memory savings
whenever the approximate solution is required at different time instances in $[0, t_f]$. Theoretical motivation for keeping the approximation space independent of the time-stepping is contained in \cite{BEHR.18}, where it is shown that the solution of the DRE lives in an invariant 
Krylov-subspace\footnote{ We also refer the reader to the recent manuscript \cite{BEHR.19.2}, which appeared on-line briefly before the first round of our revision.}.

The class of numerical methods we used
for solving the reduced DRE is described in the next section.
In the rest of this paper, we specialize the generic derivation above
to the extended and rational Krylov subspaces, which greatly outperformed polynomial spaces
both in terms of CPU time and memory requirements. More information on these spaces and their
properties are given in \cref{matral}; in particular, we discuss the generation of a real
rational Krylov basis in the presence of non-real shifts.

\section{BDF methods for the DRE} \label{sec:internal}
{The numerical solution of the small-scale DRE is a well-studied topic, see, e.g., 
\cite{Mena2007,Davison1973,Mena2018,Stillfjord2015}. Among the explored methods
are matrix generalizations of the BDF methods \cite{Mena2007, Dieci1992}, which 
are computationally appealing only for small problems. Due to the reduction strength of
rational Krylov subspaces, we expect the reduced DRE in \cref{projdre} to be
small enough to allow for efficient use of matrix-based BDF methods, which we are going to summarize
next.
For simplicity of exposition, in the rest of this section we omit the subscript in $Y_{m}$, 
and denote $Y^{(k+1)} = Y(t_{k+1})$.} If we define
\begin{equation}
\label{drerhs}
\mathbf{F}\left(Y^{(k+1)}\right) = \mathcal{T}_{m}^{T}Y^{(k+1)} + Y^{(k+1)}\mathcal{T}_{m} - Y^{(k+1)} B_{m}B_{m}^{T}Y^{(k+1)} + C_{m}^{T}C_{m},
\end{equation}
the approximation of $Y^{(k+1)}$ is given by the implicit relation
\begin{equation}
\label{eq:bdf}
Y^{(k+1)} = \sum_{i = 0}^{b-1}\alpha_{i}Y^{(k-i)} + h\beta\mathbf{F}(Y^{(k+1)}),
\end{equation}
where $h = t_{k+1} - t_{k}$ {is the stepsize and the respective $\alpha_{i}$'s and 
$\beta$ 
are the coefficients of the $b$-step BDF method for $b \leq 3$ and are given below.}
\vskip 0.1in
{\footnotesize
\begin{center}
\begin{tabular}{|c|c|c|c|c|} \hline
$p$ & $\beta$ & $\alpha_{0}$ & $\alpha_{1}$ & $\alpha_{2}$ \\ \hline
1 & 1 & 1 &  &  \\
2 & $2/3$ & $4/3$ & $-1/3$ & \\ 
3 & $6/11$ & $18/11$ & $-9/11$ & $2/11$ \\ \hline
\end{tabular}
\end{center}
}
\vskip 0.1in
%

Substituting \cref{drerhs} into \cref{eq:bdf} results in the following nonlinear matrix equation
{\footnotesize
\begin{equation*}
-Y^{(k+1)} + h\beta\left (\mathcal{T}_{m}^{T}Y^{(k+1)} + Y^{(k+1)}\mathcal{T}_{m} 
- Y^{(k+1)}B_{m}B_{m}^{T}Y^{(k+1)} + C_{m}^{T}C_{m}\right ) + \sum_{i=0}^{b-1}\alpha_{i}Y^{(k-i)} = 0,
\end{equation*}
}
which can be reformulated as the following continuous-time ARE
\begin{equation}
\label{ARE2}
\widehat{\mathcal{T}}_{m}^{T}Y^{(k+1)} + Y^{(k+1)}\widehat{\mathcal{T}}_{m} - Y^{(k+1)}\widehat{B}_{m}\widehat{B}_{m}^{T}Y^{(k+1)} + \widehat{Q}_{m} = 0.
\end{equation}
The coefficient matrices are given by
\begin{equation*}
\widehat{\mathcal{T}}_{m} = h\beta \mathcal{T}_{m} - \frac{1}{2}I_{m}, \quad 
\widehat{B}_{m} = \sqrt{h\beta} B_{m}, \quad
\widehat{Q}_{m} = h\beta C_{m}^{T}C_{m} + \sum_{i=0}^{b-1}\alpha_{i}Y^{(k-i)}.
\end{equation*}
The Riccati equation \cref{ARE2} can be solved using ``direct'' methods; see, e.g., \cite{Bini2012}.
In our experiments we used the MATLAB  solver \texttt{care} from the control systems toolbox. 
A brief sketch of the $b$-step BDF method is reported in \cref{alg:bdf}; other approaches
are discussed, e.g., in \cite{Mena2007, Lang2017}.
\begin{algorithm}

\caption{$b$-step BDF method -- BDF($b,\ell$)}
\label{alg:bdf}
\begin{algorithmic}[1]
\REQUIRE{$\mathcal{T}_{m} \in \mathbb{R}^{d \times d}$, $B_{m} \in \mathbb{R}^{d \times s}$, 
$C_{m} \in \mathbb{R}^{p \times d}$, $Z_{m} \in \mathbb{R}^{d \times q}$, 
final time $t_{f}$, number of timesteps $\ell$, initial approximations $Y^{(0)},\dots,Y^{(b-1)}$}.
\STATE{$h=t_f/\ell$, $\widehat{\mathcal{T}}_{m} = h\beta \mathcal{T}_{m} - \frac{1}{2}I_{m}$, 
$\, \widehat{B}_{m} = \sqrt{h\beta} B_{m}$}
\FOR{$ k = 0 $ \TO $\ell$}
\STATE{$\widehat{Q}_{m} = h\beta C_{m}^{T}C_{m} + \sum_{i=0}^{b-1}\alpha_{i}Y^{(k-i)}$}
\STATE{Solve $\widehat{\mathcal{T}}_{m}^{T}Y^{(k+1)} + Y^{(k+1)}\widehat{\mathcal{T}}_{m} - Y^{(k+1)}\widehat{B}_{m}\widehat{B}_{m}^{T}Y^{(k+1)} + \widehat{Q}_{m} = 0$}
\ENDFOR
\RETURN $Y^{(k)} \approx Y(t_k)$, $t_k=0, h, \ldots, t_f$
\end{algorithmic}

\end{algorithm}

We conclude the section by depicting the typical convergence behavior of 
the BDF methods in our context.
We consider an example from \cite{Mena2018}, where the $n\times n$ matrix $A$ 
stems from the spatial finite difference discretization of the following advection-diffusion equation 
\begin{equation*}
\partial_{t}w = \Delta w - 10x w_x-100y w_y, \qquad {w|_{\partial \Omega} =0}
\end{equation*}
on $\Omega = (0,1)^{2}$ with homogeneous Dirichlet boundary conditions. The choices of $B \in \mathbb{R}^{n \times 1}$ and $C \in \mathbb{R}^{1 \times n}$ are given binomially as described in \cite{Mena2018}. 
The initial condition is taken to be the zero matrix, that is $Z = \mathbf{0}_{n \times 1}$. 
We compare the obtained solution with a ``reference'' numerical solution $Y_{ref}(t)$ computed
by an accurate but expensive method (the MATLAB function \texttt{ode23s} in our experiments), so that
$n$ is kept small, $n=49$.
The convergence behavior for $b = 1,2,3$ and $\ell$ timesteps, with $\ell = 10,100,1000$ is displayed in \cref{bdfs}.
The left plot shows the error $\|Y(t)-Y_{ref}(t)\|$ as a function of $t$, for different values
of $\ell$. The right plot shows the evolution of the (1,1) component of the solution throughout
the time span for the most accurate choice of BDF method, compared with that of the reference solution.
\begin{figure}[htb!]    
\begin{center}   
    \includegraphics[width=0.47\textwidth]{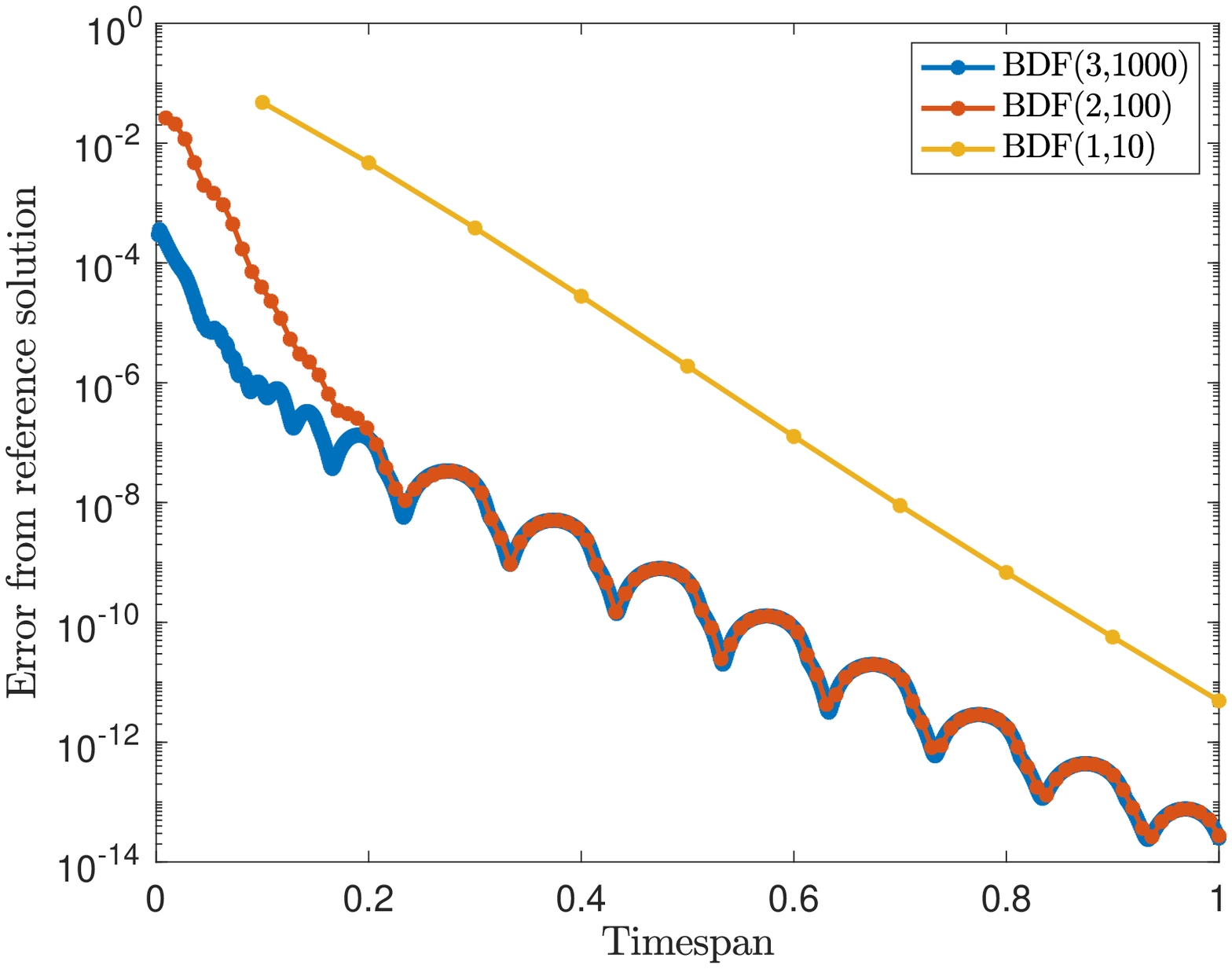}
    \hspace{0.1cm}
    \includegraphics[width=0.47\textwidth]{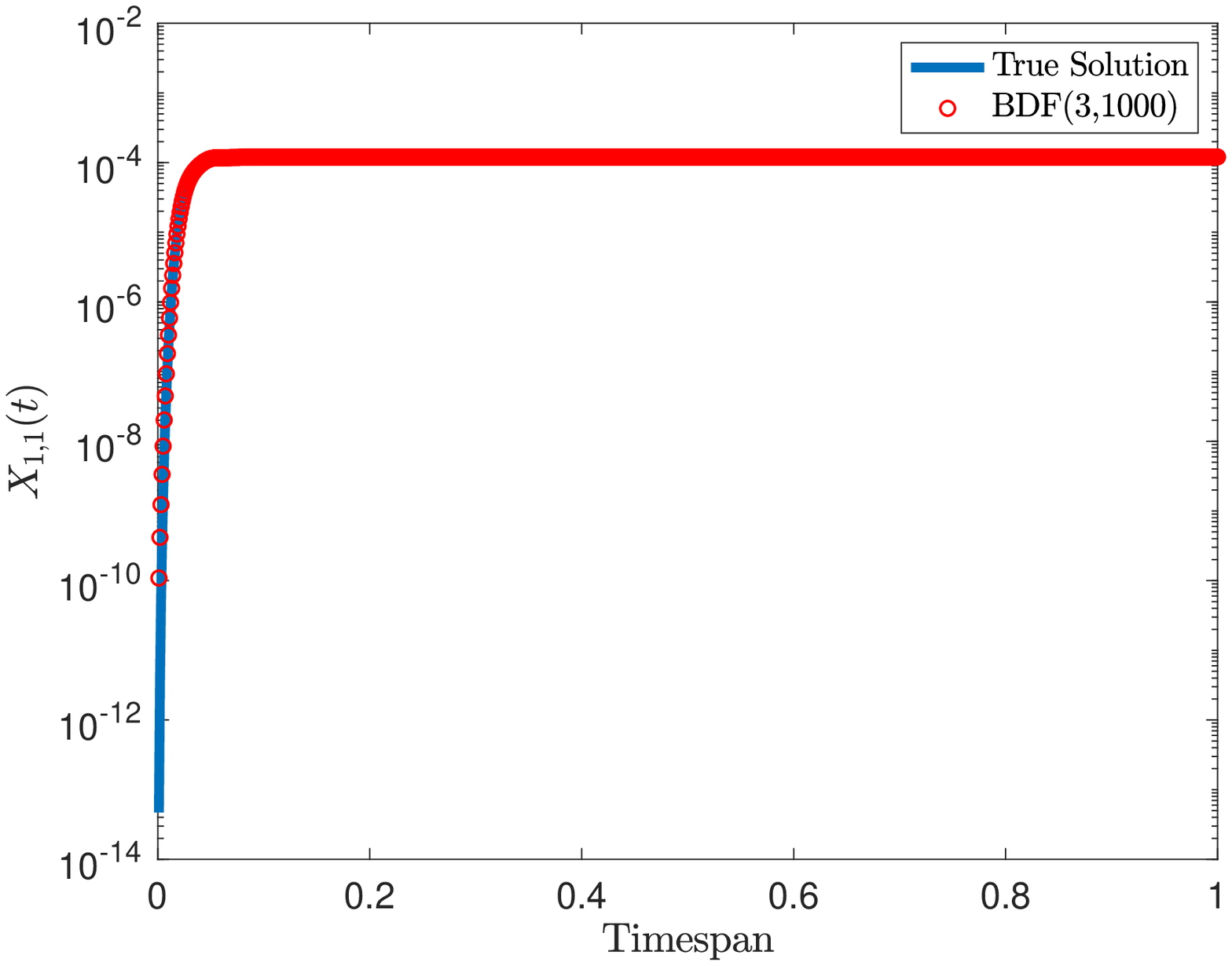}
      \caption[Examples of sparse matrices]{Typical convergence behavior of BDF methods (left) and 
evolution of the $X_{1,1}$ component of the reference and BDF($3,1000$) solution (right).}
    \label{bdfs}
    \end{center}
\end{figure}
These plots illustrate that we cannot expect an overall high accuracy of the projection
method as long as the reduced differential equation is not solved with sufficiently good accuracy.
The importance of this is discussed in more detail in the following section.

\section{Stopping criterion and the complete algorithm} \label{sec:krylovdre}
To complete the reduction algorithm of section \ref{sec:projection}, we need to
introduce a stopping criterion. We found that it is crucial to take into
account the accuracy of the numerical method employed
to solve the reduced DRE, as discussed in section \ref{sec:internal}.

To derive our stopping criterion we were inspired by those in \cite{Koskela2018, Guldogan2017}, 
however, we made some important modifications.
In both cited references, the authors assume 
that the inner problem \cref{projdre} is solved exactly, which is not true in general. We thus 
consider that  
the numerical method solves the reduced problem with residual matrix
$R_{m}^{(I)}(t) := \dot{Y}_{m}(t) - \mathbf{F}(Y_{m}(t))$, so that
the final DRE residual can be split into two components.
\begin{proposition}\label{thm:res}
Let $X_{m}(t) = V_{m}Y_{m}(t)V_{m}^{T}$ be the Krylov-based approximate solution after 
$m$ iterations, where $Y_{m}(t)$ approximately solves the reduced problem \cref{projdre}. With
the previous notation, the residual 
matrix
$\mathcal{R}_{m}(t)= \dot{X}_{m}(t) - \mathbf{F}(X_{m}(t))$ satisfies
\begin{equation}
\label{eq:timeres}
\|\mathcal{R}_{m}(t)\|_F^{2} = \|R_{m}^{(I)}(t)\|_F^{2} + 2\| R_{m}^{(O)}(t)\|_F^{2},
\end{equation}
where $R_{m}^{(I)}(t) = \dot{Y}_{m}(t) - \mathbf{F}(Y_{m}(t))$ 
and $R_{m}^{(O)}(t) = \tau_{m}^{T}Y_{m}(t)$ with $\tau_{m}$ as in (\ref{genericrelation}).
\end{proposition}

\begin{proof}
Substituting \cref{projection} into the residual $\mathcal{R}_{m}(t)$ we obtain
 \begin{equation}
\begin{split}
\label{resproof1}
\mathcal{R}_{m}(t) = \mathcal{V}_{m}\dot{Y}_{m}(t)\mathcal{V}_{m}^{T} &- A^{T}\mathcal{V}_{m}Y_{m}(t)\mathcal{V}_{m}^{T} - \mathcal{V}_{m}Y_{m}(t)\mathcal{V}_{m}^{T}A\\ &+ \mathcal{V}_{m}Y_{m}(t)\mathcal{V}_{m}^{T}BB^{T}\mathcal{V}_{m}Y_{m}(t)\mathcal{V}_{m}^{T} - C^{T}C.
\end{split}
\end{equation}
Since $C^T$ belongs to range$({\cal V}_m)$, we can write $C^T={\cal V}_mC_m^{T}$.
Using~\cref{genericrelation}, we get
\begin{equation*}
\begin{split}
\!\!\!\mathcal{R}_{m}(t) = \mathcal{V}_{m}\dot{Y}_{m}(t)\mathcal{V}_{m}^{T} &- (\mathcal{V}_{m}\mathcal{T}_m^{T} +  
\nu_{m+1}\tau_{m}^{T})Y_{m}(t)\mathcal{V}_{m}^{T}
 - \mathcal{V}_{m}Y_{m}(t)( \mathcal{T}_{m}\mathcal{V}_{m}^{T} + \tau_{m}\nu_{m+1}^{T})\\ 
&+ \mathcal{V}_{m}Y_{m}(t)\mathcal{V}_{m}^{T}BB^{T}\mathcal{V}_{m}Y_{m}(t)\mathcal{V}_{m}^{T} 
     - {\cal V}_mC_m^{T} C_m{\cal V}_m^{T}.
\end{split}
\end{equation*}
Since $\mathcal{V}_{m+1} = [\mathcal{V}_{m}, \nu_{m+1}]$, we can write
$\mathcal{R}_{m}(t) = \mathcal{V}_{m+1}\mathcal{J}_{m}(t)\mathcal{V}_{m+1}^{T}$, where
{\small
\begin{equation*}
\mathcal{J}_{m}(t) = \left[
\begin{array}{c|c}
 \dot{Y}_{m}(t) - \mathcal{T}_{m}^{T}Y_{m}(t) - Y_{m}(t)\mathcal{T}_{m} + Y_{m}(t)B_{m}B_{m}^{T}Y_{m}(t) - C_m^{T} C_m   &      Y_{m}(t)\tau_{m} \\ \hline
\tau_{m}^{T}Y_{m}(t)  & \mathbf{0}
\end{array}
\right].
\end{equation*} }
Let $R_{m}^{(I)}(t)$ be the residual of the numerical ODE inner solver. Then
\begin{equation*}
\mathcal{J}_{m}(t) = \left[
\begin{array}{c|c}
R_{m}^{(I)}(t)    &      Y_{m}(t)\tau_{m} \\ \hline
\tau_{m}^{T}Y_{m}(t)  & \mathbf{0}
\end{array}
\right].
\end{equation*}
Since the columns of $\mathcal{V}_{m+1}$ are orthonormal,
\begin{equation*}
\begin{split}
\|\mathcal{R}_{m}(t) \|_{F}^2 &= \|\mathcal{V}_{m+1}\mathcal{J}_{m}(t)\mathcal{V}_{m+1}^{T}\|_{F}^2
= \|\mathcal{J}_{m}(t)\|_{F}^2\\
&= \mbox{Tr}\left(R_{m}^{(I)}(t)^{T}R_{m}^{(I)}(t) + 2(Y_{m}(t)\tau_{m})(\tau_{m}^{T}Y_{m}(t)) \right),
\end{split}
\end{equation*}
that is, $\|\mathcal{R}_{m}(t)\|_F^{2} = \|R_{m}^{(I)}(t)\|_F^{2} + 2\| \tau_{m}^{T}Y_{m}(t)\|_F^{2}$,
and the result follows. 
\end{proof}
The expression for ${\cal J}_m(t)$ emphasizes that at each iteration $m$
the matrix $Y_m(t)$ is the exact solution of
$$
\dot{Y}_{m}(t) - \mathcal{T}_{m}^{T}Y_{m}(t) - 
Y_{m}(t)\mathcal{T}_{m} + Y_{m}(t)B_{m}B_{m}^{T}Y_{m}(t) - C_m C_m^{T} - R_{m}^{(I)}(t) = 0 .
$$
Hence, as long as $\|R_{m}^{(I)}(t)\|_{F}$ is not very small, the increase of $m$ aims at more and more
accurately approximating a ``nearby'' differential problem to the truly projected one, with 
a term $R_{m}^{(I)}(t)$ that varies with $m$. Hence, $X_m(t) = {\cal V}_m Y_m {\cal V}_m^T$ is
an approximation not to $X(t)$, but to the solution of a differential problem with an
additional term whose projection onto the space is $R_{m}^{(I)}(t)$.

Proposition \ref{thm:res} also implies that 
we cannot expect an overall small residual norm if either of the two partial
residual norms $\|R_{m}^{(I)}(t)\|_F$, $\|R_{m}^{(O)}(t))\|_F$
is not small. 
In particular, we observe that the two residuals can be made small independently. 
%
Therefore we propose the following practical strategy:
\begin{enumerate}[i)]
\item[(i)] Run the algorithm as presented, with a low-order cheap 
ODE inner solver (i.e., BDF(1, $\ell$) with $\ell$ relatively small) and use ${R}_{m}^{(O)}(t)$ in the stopping criterion;
\item[(ii)] Once completed step (i) after $\widehat{m}$ iterations, 
use the matrices $T_{\widehat{m}}, C_{\widehat{m}}, B_{\widehat{m}}$ and $Z_{\widehat{m}}$ to
 refine the ODE inner solution by using a higher-order ODE solver for the projected system.
\end{enumerate}
The final matrix $Y_{\widehat m}(t)$ obtained in step (ii) will provide a more accurate solution matrix
than what would have been obtained at the end of step (i). We emphasize that 
\emph{any} ODE method for small and medium scale DREs could be used at steps (i) and (ii). 
Our choice of BDF(1, $\ell)$ is due to its good trade-off between accuracy and computational effort;
other approaches could be considered.

To complete the description of the stopping criterion, we recall that $R_{m}^{(O)}(t)$ depends on $t$, so
that we need to estimate the integral over the whole time interval by means of a quadrature formula, that is
 \begin{equation} \label{cheapy}
\int_0^{t_f} \|R_{m}^{(O)}(\gamma)\|_F d\gamma \approx \sum_{j = 1}^{\ell} \frac{t_f}{\ell} \|R_{m}^{(O)}(t_j)\|_F  =:
\rho_m
\end{equation}
where the interval $[0,t_f]$ has been divided into $\ell$ intervals with nodes $t_j$.

The overall algorithm\footnote{A Matlab implementation of both algorithms 
will be made available upon publication of this work.}
based on the rational Krylov subspace method is
reported in \cref{alg:rksmdre}, while the algorithm based on the extended method is
postponed to \cref{appb}.
Several implementation issues of \cref{alg:rksmdre} are
also described in \cref{matral}, such as the use of a real basis in case of complex
shifts $s_j$ in the basis construction.

\begin{algorithm}
\caption{RKSM-DRE}
\label{alg:rksmdre}
\begin{algorithmic}[0]
\REQUIRE{$A \in \mathbb{R}^{n \times n}$, $B \in \mathbb{R}^{n \times s}$, $C \in \mathbb{R}^{p \times n}$, $Z \in \mathbb{R}^{n \times q}$, $tol$, $t_{f}$, $\ell$, $\mathbf{s_{0}} = \{s_{0}^{(1)}, s_{0}^{(2)}\}$}
\STATE{(i) Perform reduced $QR$: $[C^{T},\, Z] = V_{1}\Lambda_{1}$}
\STATE{\hskip 0.2in Set $\mathcal{V}_{1} \equiv V_{1}$}
	\STATE{\hskip 0.2in {\bf for}  $ m = 2,3 \dots$}
	\STATE{\hskip 0.4in  Compute the next shift and add it to $\mathbf{ s_{0} }$}
	\STATE{\hskip 0.4in Compute the next {\it real} basis block $V_{m}$} 
	\STATE{\hskip 0.4in Set $\mathcal{V}_{m} = [\mathcal{V}_{m-1}, V_{m}]$}
	\STATE{\hskip 0.4in Update $\mathcal{T}_{m} = \mathcal{V}_{m}^{T}A\mathcal{V}_{m}$ and 
$B_{m} 	= \mathcal{V}_{m}^{T}B$, $Z_{m} = \mathcal{V}_{m}^{T}Z$  and $C_{m} = C\mathcal{V}	_{m}$}
	\STATE{\hskip 0.4in  Integrate \cref{projdre} from 0 to $t_f$ using BDF($1,\ell$)}
	\STATE{\hskip 0.4in Compute $\rho_m$ using \cref{cheapy} where $\tau_{m}^{T} = G_{m}^{T}$}	 
		\STATE{\hskip 0.4in {\bf if} $\rho_m < tol$}
		\STATE{\hskip 0.6in \textbf{go to} (ii)}
		\STATE{\hskip 0.4in {\bf end if}}
\STATE{\hskip 0.2in{\bf end for}}
\STATE{(ii) Refinement: solve \cref{projdre} with a more accurate integrator}
\STATE{Compute $Y_{m}(t_j) = \widehat{Y}_{m}(t_j)\widehat{Y}_{m}(t_j)^{T}$, $j=1, \ldots, \ell$ using the truncated SVD}
\RETURN $\mathcal{V}_{m} \in \mathbb{R}^{n \times m(p+q)}$ and $\ell$ 
factors $\widehat{Y}_{m}(t_j) \in \mathbb{R}^{m(p+q) \times r}$, $j=1, \ldots, \ell$ 
\end{algorithmic}
\end{algorithm}

\section{Stability analysis and error bounds}
\label{sec:stab}
In this section, we provide a few results on the spectral
and convergence properties of the obtained approximate solution
We first inspect some properties of the 
asymptotic matrix solution, which solves the algebraic Riccati equation.
Then we propose a bound for the error matrix, in an appropriate functional norm.

\subsection{Properties of the (steady state) algebraic Riccati equation}
\label{sec:algstab}
Properties associated with the algebraic Riccati equation -- as
asymptotic solution to the DRE -- are well known in 
linear quadratic optimal control, see, e.g., \cite{Kwakernaak1972, CORLESS2003}.
%
%
In particular, classical uniqueness and stabilization properties 
of the solution (see, e.g., \cite[Lemma 12.7.2]{CORLESS2003}),
can directly be extended to the reduced DRE \cref{projdre}.
\begin{corollary} \label{cor:reduced}
Let $({\cal T}_{m}, B_{m}, C_{m})$ be stabilizable and detectable system. Let $Y_{m}(t)$ be the solution of \cref{projdre} at time $t$ and let $Y_{m}^{\infty} = \lim_{t \rightarrow \infty} Y_{m}(t)$. Then $Y_{m}^{\infty}$ is the unique symmetric nonnegative definite solution and the only stabilizing solution to the (reduced) algebraic Riccati equation 
\begin{equation} \label{projare}
0 = {\cal T}_{m}^{T}Y_{m}^{\infty} + Y_{m}^{\infty}{\cal T}_{m} - Y_{m}^{\infty}B_{m}B_{m}^{T}Y_{m}^{\infty} + C_{m}^{T}C_{m}.
\end{equation}
Moreover, if the pair $(C_{m},{\cal T}_{m})$ is observable, $Y_{m}^{\infty}$ is strictly positive definite.
\end{corollary}

We notice that the stabilizability and detectability properties
of $({\cal T}_{m}, B_{m}, C_{m})$ are not necessarily implied by those on $(A, B,C)$.
Nevertheless, it is shown in \cite{Simoncini2016} that if 
there exists a feedback matrix $K$, such that the linear dynamical system $ \dot{x}=(A - BK)x$ is dissipative, then the 
pair $({\cal T}_{m}, B_{m})$ is stabilizable. 
{A similar result can be formulated for the 
detectability of $(C_{m}, {\cal T}_{m})$, since by duality reasoning, 
$(C_{m}, {\cal T}_{m})$ is detectable if $({\cal T}_{m}^{T},C_{m}^{T})$ is stabilizable. The question regarding the existence of such a feedback matrix, with respect to $A$ and $B$ ($A^T$ and $C^T$), is addressed in \cite{Guglielmi2019}.

With these results, we can relate the asymptotic solution of the
original and projected problems.
Let $X_{m}(t) = {\cal V}_{m}Y_{m}(t){\cal V}_{m}^{T}$ and 
$X_{m}^{a} = {\cal V}_{m}Y_{m}^{\infty}{\cal V}_{m}^{T}$ respectively be approximate solutions to 
\cref{dre,are} by a projection onto $\mbox{range}({\cal V}_{m})$. 
If there exist matrices $K$ and $L$ such that the systems $\dot{x} = (A - BK)x$ and $\dot{x} = (A^{T} - C^{T}L)x$ are dissipative, then
\begin{equation}
\lim_{t \rightarrow \infty}X_{m}(t) = {\cal V}_{m}\lim_{t \rightarrow \infty}Y_{m}(t){\cal V}_{m}^{T}
 = {\cal V}_{m}Y_{m}^{\infty}{\cal V}_{m}^{T}
 = X_{m}^{a},
\end{equation}
that is, $X_{m}^{a}$ is the steady state solution of $X_{m}(t)$ when projected onto the same basis.

Under the hypotheses that $(A,B,C)$ is a stabilizable and detectable system, there
exists a unique non-negative and stabilizing solution 
$X_\infty$ to \cref{are} (see, e.g., \cite[Theorem 5]{Kucera1973}).
In \cite{Simoncini2016} a bound was derived for the error 
$X_{\infty} - X_{m}^{a}$ in terms of the matrix residual norm.
Here we complete the argument by stating that in exact arithmetic and if
the whole space can be spanned, the obtained approximate solution equals $X_\infty$.

 \begin{proposition}\label{propstar}
Suppose $(A,B,C)$ is stabilizable and detectable. 
Assume it is possible to determine $m_{*}$ such that $\mbox{dim}(range({\cal V}_{m_{*}})) = n$, and 
let $X_{m_{*}}^{a} = {\cal V}_{m_{*}}Y_{m_{*}}^{\infty}{\cal V}_{m_{*}}^{T}$ be the 
obtained approximate solution of \cref{are} after $m_{*}$ iterations. 
Then, $X_{m_{*}}^{a} = X_\infty$.
\end{proposition}

\begin{proof}
Since ${\cal V}_{m_{*}}$ is square and orthogonal the projected ARE is given by
$$
0 = {\cal V}_{m_{*}}^{T}A^{T}{\cal V}_{m_*}Y_{m_*}^{\infty} + 
Y_{m_*}^{\infty}{\cal V}_{m_*}^{T}A{\cal V}_{m_*} - 
Y_{m_*}^{\infty}{\cal V}_{m_*}^{T}BB^{T}{\cal V}_{m_*}Y_{m_*}^{\infty} 
+ {\cal V}_{m_{*}}^{T}C^{T}C{\cal V}_{m_{*}} .
$$
From $(A, B, C)$ stabilizable and detectable it follows that
$({\cal V}_{m_*}^T A {\cal V}_{m_*}, {\cal V}_{m_*}^T B, C{\cal V}_{m_*})$ is also
stabilizable and detectable, so that $Y_{m_*}^\infty \ge 0$ and stabilizing.
Multiplying by ${\cal V}_{m*}$ (by ${\cal V}_{m*}^T$) from the left (right), we obtain
$
0 = A^{T}X_{m_{*}}^{a} + X_{m_{*}}^{a}A  - X_{m_{*}}^{a}BB^{T}X_{m_{*}}^{a} + C^{T}C,
$
that is, $X_{m_{*}}^{a} \ge 0$ is a solution to the original ARE. Since $X_{\infty}$ is the unique
nonnegative definite solution, it must be $X_{m_{*}}^{a}=X_\infty$.
\end{proof}

\subsection{Error bound for the differential Riccati equation}
In this section we derive a bound for the maximum error obtained by the reduction process, 
in terms of the residual 
\begin{equation}\label{residual}
\begin{split}
R_{m}(t) = A^{T}X_{m}(t) + X_{m}(t)A - X_{m}(t)BB^{T}X_{m}(t) + C^{T}C - \dot{X}_{m}(t).
\end{split}
\end{equation}
Note that $R_m(t)$ is the residual matrix with respect to the exact solution
of the reduced differential problem, that is, it also includes the discretization error.
A similar bound on the error has been derived 
for the nonsymmetric DRE in \cite{Angelova2018}, which used matrix perturbation techniques
from \cite{Konstantinov1991}. 

\begin{proposition}
For $t\in[0, t_f]$ let ${\cal E}_m(t) = X(t) - X_m(t)$ and assume that 
${\cal A}(t):=A - BB^{T}X(t)$ is stable for all $t\in[0, t_f]$. Denote
$$
{\color{black}\nu := 
\max_{t \in [0,t_{f}]}\left\{ \int_0^t \|\Phi_{{\mathcal{A}}^T}(t,s)\| \, \| \Phi_{\mathcal{A}}(t,s) \| ds \right\},
}
$$
where $\!\Phi_{{\cal A}}\!$ is the state-transition matrix satisfying
$\frac{\partial{\Phi}_{{\cal A}}(t,s)}{\partial t}\!=\!{\cal A}(t)\Phi_{{\cal A}}(t,s),
\Phi_{\cal A}(s,s)=I$. 
%
 
If $4\nu^{2}\| B \|^{2}\|R_{m}\|_{\infty_t} < 1$, then
$$
\|{\cal E}_m\|_{\infty_t} \le 2\nu \|R_m\|_{\infty_t},
$$
where $\|L\|_{\infty_t} = \max_{t\in[0,t_{f}]} \|L(t)\|$ for any continuous matrix function $L(t)$. 
\end{proposition}

\begin{proof}
By subtracting \cref{residual} from \cref{dre} and manipulating terms we obtain
$$
{\dot {\cal E}}_{m}(t) = (A-BB^{T}X(t))^T {\cal E}_{m}(t) + {\cal E}_{m}(t) (A-BB^{T}X(t)) + {\cal E}_{m}(t) BB^T {\cal E}_{m}(t) + R_{m}(t),
$$
with ${\cal E}_{m}(0) = 0$. Therefore, by the variation of constants formula (see, e.g., \cite{Kucera1973})
$$
{\cal E}_{m}(t) = \int_0^t \Phi_{{\cal A}^{T}}(t,s) \left(R_{m}(s) +
 {\cal E}_{m}(s) BB^T {\cal E}_{m}(s) \right) \Phi_{{\cal A}}(t,s)  ds.
$$ 
Taking norms yields
$$
\| {\cal E}_{m}(t) \|_{\infty_t} \leq \max_{t \in [0,t_f]} 
\int_0^t \| \Phi_{{\mathcal{A}}^T}(t,s) \|\, \| \Phi_{\mathcal{A}}(t,s) \| \left( \| R_{m}(s)\| + \|{\cal E}_{m}(s)\|^2 \|B\|^2\right)  ds,
$$ 
so that
$
\| {\cal E}_{m}(t) \|_{\infty_t} \leq  \nu \left( \| R_{m}(t)\|_{\infty_t} + \|{\cal E}_{m}(t)\|_{\infty_t}^2 \|B\|^2\right).
$
Solving this quadratic inequality yields
$$
\| {\cal E}_{m}(t) \|_{\infty_t} \leq \frac{1 - \sqrt{1 - 4\nu^{2}\| B \|^{2}\|R_{m}\|_{\infty_t}}}{2\nu \| B \|^2}.
$$
The result follows from multiplying and dividing by 
$(1 + \sqrt{1 - 4\nu^{2}\| B \|^{2}\|R_{m}\|_{\infty_t}})$
and noticing that at the denominator this quantity can be bounded from below by 1.
\end{proof}

We conclude with a remark on the intuitive fact that if the approximation space spans the whole
space, the obtained solution by projection necessarily coincides with the
sought after solution of the DRE.

\begin{remark}
If it is possible to determine $m_{*}$  such that  $\mbox{dim}({\cal V}_{m_{*}}) = n$,
then the approximate solution $X_{m_{*}}(t)$ coincides with $X(t)$ for all $t \geq 0$.
Indeed, let us write
  $X_{m_*}(t)= {\cal V}_{m_{*}}Y_{m_{*}}(t){\cal V}_{m_{*}}^{T}$, where
${\cal V}_{m_{*}}$ is square and orthogonal. The reduced DRE is given by
$$
\dot{Y}_{m_{*}}= 
{\cal V}_{m_{*}}^{T}A^{T}{\cal V}_{m_{*}}Y_{m_{*}} + 
Y_{m_{*}}{\cal V}_{m_{*}}^{T}A{\cal V}_{m_{*}} - Y_{m_{*}}{\cal V}_{m_{*}}^{T}BB^{T}{\cal V}_{m_{*}}Y_{m_{*}}
+ {\cal V}_{m_{*}}^{T}C^{T}C{\cal V}_{m_{*}} 
$$
with $Y_{m_*}=Y_{m_*}(t)$.
Multiplying by ${\cal V}_{m*}$ (by ${\cal V}_{m*}^T$) from the left (right), we obtain
\begin{equation*}
\dot{X}_{m_{*}}(t) = A^{T}X_{m_{*}}(t) + X_{m_{*}}(t)A  - X_{m_{*}}(t)BB^{T}X_{m_{*}}(t) + C^{T}C ,
\end{equation*}
hence, $X_{m_{*}}(t) \geq 0$ is a solution of \cref{dre}.
Since $X(t)$ is the unique nonnegative definite solution of \cref{dre} for any $X_{0} \geq 0$
(see, e.g., \cite{Kucera1973}), then $X_{m_{*}}(t)=X(t)$ for $t\geq 0$. $\,\,\square$
\end{remark}

\section{Numerical experiments}
\label{sec:main}
In this section we report on our numerical experience with the developed techniques.
We consider two artificial symmetric and nonsymmetric model problems, as well as three (of which two are nonsymmetric)
standard benchmark problems. Information about the considered data is contained in \cref{data}. 
For the first two datasets displayed in \cref{data}, the matrix $A$ stems from the 
finite difference discretization with homogenous Dirichlet boundary conditions on the unit square and unit cube,
 respectively. 
The first matrix ({\sc sym2d}) comes from the finite difference discretization of the two-dimensional Laplace operator
in the unit square with homogeneous boundary conditions,
while the second matrix ({\sc nsym3d}) stems from the finite difference
discretization of the three-dimensional differential operator 
\begin{equation*}
\mathcal{L}(u) = e^{xy}(u_{x})_{x} + e^{xy}(u_{y})_{y} + (u_{z})_{z} + (1+x)e^{-x}u_{x} + y^{2}u_{y} + 10(x+y)u_{z} ,
\end{equation*}
in the unit cube, with homogeneous boundary conditions.
For both datasets, the matrices $B, C$ and $Z$ are selected randomly with normally distributed entries. The realizations of the random matrices are fixed for both examples using the MATLAB command {\tt rng}:
for $B, C$ and $Z$ we use {\tt rng}(7), {\tt rng}(2) and {\tt rng}(3), respectively. 
The following two datasets ( {\sc chip} and {\sc flow}) are 
taken from \cite{OB2003}, and all coefficient matrices ($\widehat{A}$, $\widehat{B}$, $\widehat{C}$ and $\widehat{E}$) are contained in the datasets, which stem from the dynamical system
$$
\widehat{E}\dot{\widehat{x}} = \widehat{A}\widehat{x} + \widehat{B}u, 
\qquad \widehat{y} = \widehat{C}\widehat{x}.
$$

Since $\widehat{E}$ is diagonal and nonsingular, it is incorporated as 
$A = \widehat{E}^{-\frac 1 2}\widehat{A}\widehat{E}^{-\frac 1 2}$, while $\widehat{B}$ and $\widehat{C}$ are 
updated accordingly to form $B$ and $C$.

The final considered dataset { \sc(rail)} stems from a semi-discretized heat transfer problem for optimal cooling of steel profiles\footnote{Data available at {\tt http://modelreduction.org/index.php/Steel\textunderscore Profile} } \cite{Benner.05}. 
We consider the largest of the four available discretizations 
(file {\tt rail\textunderscore79841\textunderscore c60} containing $\widehat{A}, \widehat{B}, \widehat{C}$ and
$\widehat{E}$) with $n = 79841$. 
The symmetric and positive definite mass matrix $\widehat{E}$ 
has a sparsity pattern very similar to $\widehat{A}$. 
Both matrices are therefore reordered by the same approximate minimum degree 
({\sc rksm-dre}) or reverse Cuthill-McKee ({\sc eksm-dre}) permutation to limit fill-in. 
The state-space transformation is done using the Cholesky factorization of $\widehat{E}$.
More precisely, let $\widehat{E} = \widehat{E}_L\widehat{E}_L^T$ with $\widehat{E}_L$ lower
triangular, and consider the transformed state $x = \widehat{E}_L^T\widehat{x}$. Then 
$$
\dot{x} = Ax + Bu, \qquad y = Cx,
$$
with $A = \widehat{E}_L^{-1}\widehat{A}\widehat{E}_L^{-T}$, 
$B = \widehat{E}_L^{-1}\widehat{B}$ and $C = \widehat{C}\widehat{E}_{L}^{-T}$. 
These matrices are {\em never} explicitly formed, 
rather they are commonly applied implicitly by solves with the factor $\widehat{E}_L$ at each iteration;
see, e.g., \cite{Druskin2011, Simoncini2007}. 

The initial low-rank factors are selected as the zero vector for 
{\sc flow}, {$Z = \sin{g}$} for {\sc chip} and {$Z = \cos{g}$} for {\sc rail}, where 
$g \in \mathbb{R}^{n \times 1}$ is a vector with entries in $[0,2\pi]$. Other sufficiently
general choices were tried during our numerical investigation however results did
 not significantly differ from the ones we report.

\begin{table}[htb!]\caption{Relevant information concerning the experimental data}\label{data}
\centering
\begin{tabular}{|l|r|r|r|r|r|r|r|}
\hline
Name            & $n$      & $p$/$s$/$q$ & $||A||_{F}$ & $||B||_{F}$ & $||C||_{F}$ & $||Z||_{F}$ & $||E||_{F}$ \\ \hline
\sc sym2d & $640000$ & $5/1/1$            & $3.6 \cdot 10^{3}$  & $8.0 \cdot 10^{2}$  & $1.8 \cdot 10^{3}$  & $8.0 \cdot 10^{2}$ & $8 \cdot 10^2$  \\ 
\sc nsym3d       & $64000$ & 6/1/3            & $2.0 \cdot 10^{3}$  & $2.5 \cdot 10^{2}$  & $6.2 \cdot 10^{2}$  & $2.8 \cdot 10^{2}$ & $2.5 \cdot 10^2$  \\
\hline

Name            & $n$      & $p$/$s$/$q$ & $||\widehat{A}||_{F}$ & $||\widehat{B}||_{F}$ & $||\widehat{C}||_{F}$ & $||Z||_{F}$ & $||\widehat{E}||_{F}$ \\ \hline
\sc chip           & $20082$           & 5/1/1            & $2.2 \cdot 10^{6}$  & $1.7 \cdot 10^{2}$  & $3.3 \cdot 10^{4}$  & $1.0 \cdot 10^{2}$     & $2 \cdot 10^{-4}$           \\ 
\sc flow            & $9669$            & 5/1/1            & $4.5 \cdot 10^{6}$  & $2.0 \cdot 10^{4}$  & $1.2 \cdot 10^{3}$  & $-$             & $6.8 \cdot 10^{0}$        \\ 
\sc rail            & $79841$            & 7/6/1            & $7 \cdot 10^{-3}$  & $1 \cdot 10^{-7}$  & $6.2 \cdot 10^{0}$  & $1.9 \cdot 10^{2}$        & $8 \cdot 10^{-4}$         \\ \hline
\end{tabular}
\end{table}

\vskip 0.1in
{\it Performance of the projection methods.}
We first investigate the convergence behavior of the outer solver. 
{The quantity we monitor in our stopping criterion is the backward error in an integral norm given by
\begin{equation} \label{backerror}
 \frac{  \rho_m}{t_{f}\|C\|_F^2 + 2\xi_m + \psi_m},
\end{equation}
 with $\rho_m$ as in \cref{cheapy} and
$$
 \xi_{m} = \int_{0}^{t_f}\|A^{T}{\cal V}_{m}Y_{m}(\gamma)\|_{F} \, d\gamma \quad 
\mbox{and} \quad \psi_{m} = \int_{0}^{t_f}\|Y_{m}(\gamma)V_{m}^{T}B\|_{F}^{2} d\gamma.
$$ 
The integrals are approximated by a quadrature formula in a similar fashion to \cref{cheapy},
and we note that $\xi_{m}$ can be cheaply computed by using the Arnoldi-type relation.
}

For all datasets, the stopping tolerance was chosen as $10^{-7}$. For the first four datasets, $t_f = 1$ and {\sc bdf}(1,10) is used as inner solver. For {\sc rail}, $t_f = 4500$ (see e.g., \cite{Benner.05} for further details about the setting) and {\sc bdf(1,45)} is used as inner solver. \crefrange{fig:ex1}{fig:ex5} 
display the convergence of the 
rational Krylov subspace method (\cref{alg:rksmdre}, {\sc rksm-dre}) and of the extended Krylov subspace method 
(\cref{alg:eksmdre}, {\sc eksm-dre}). The left plots report the history of the backward error as the 
approximation space dimension increases, while the right plots display the same history versus the total
computational time (in seconds) as the iterations proceed. We notice that the cost of the refinement step is not
taken into account in these first tests.
\begin{figure}[htb!]       
    \includegraphics[width=.49\textwidth]{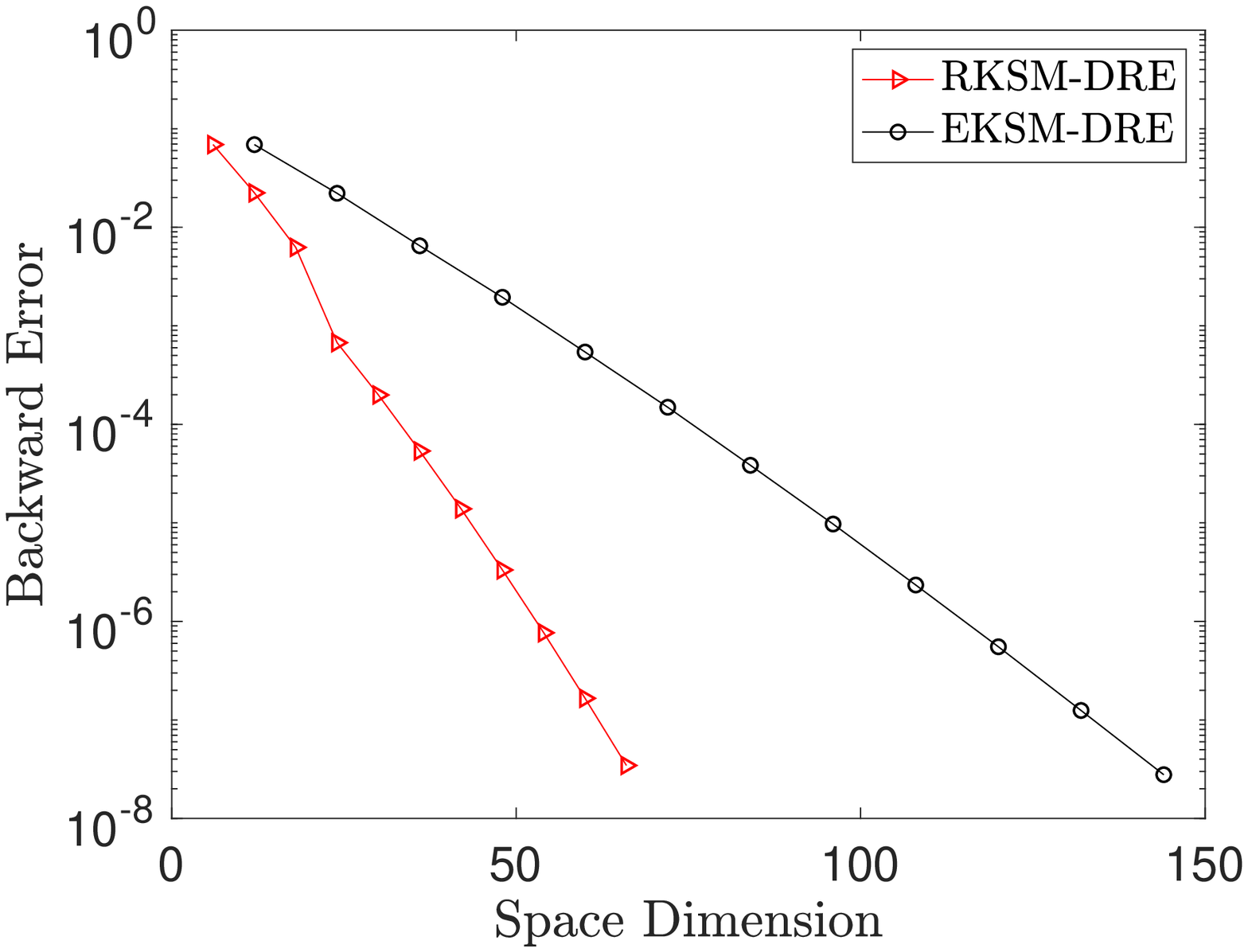}
    \hspace{0.1cm}
    \includegraphics[width=.49\textwidth]{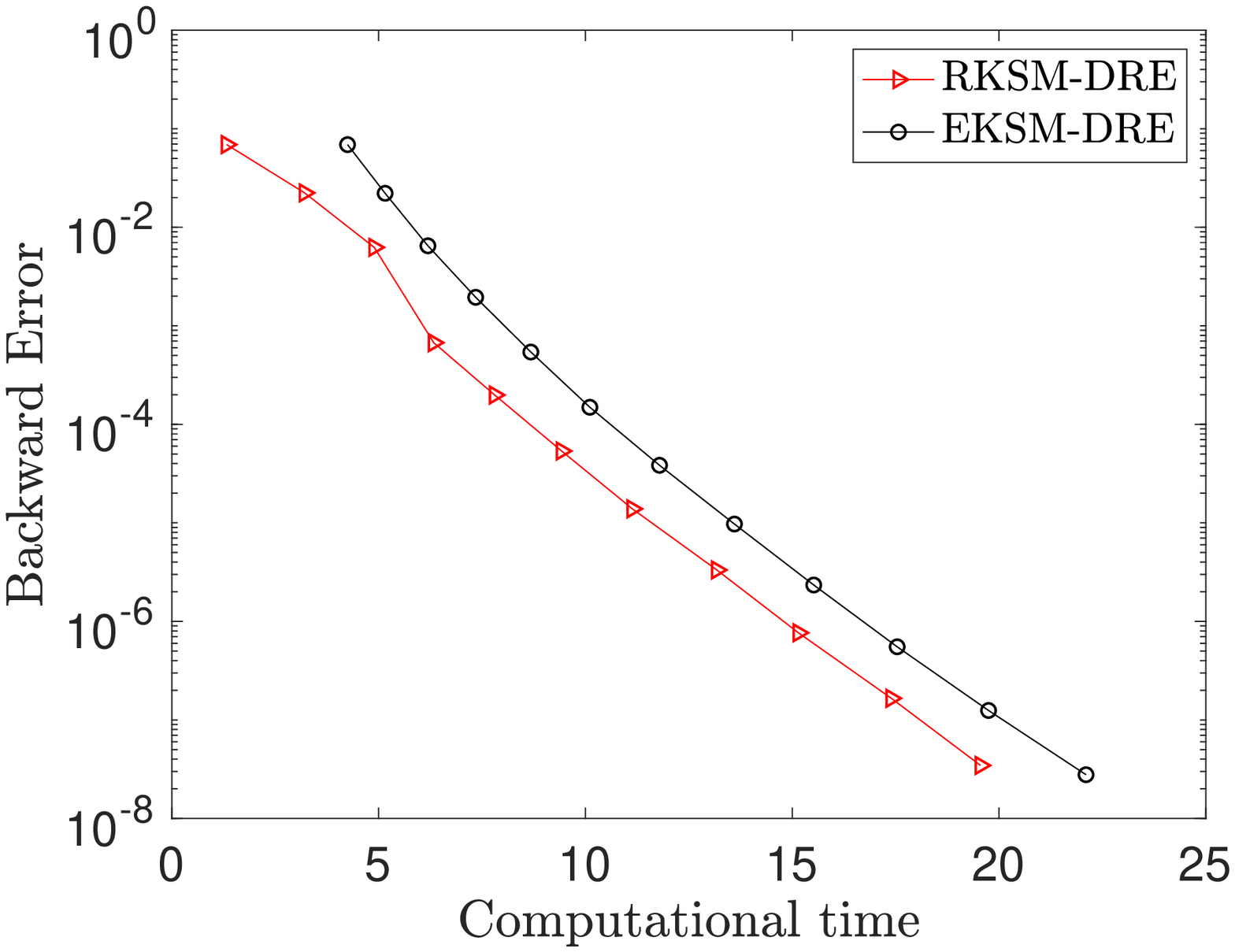}
    \caption[Examples of sparse matrices]{{\sc sym2d}: Convergence history for 
{\sc eksm-dre} and {\sc rksm-dre}. Left: backward error versus space dimension. Right:
backward error versus computational time.}
    \label{fig:ex1}
\end{figure}
 \begin{figure}[htb!]       
    \includegraphics[width=.49\textwidth]{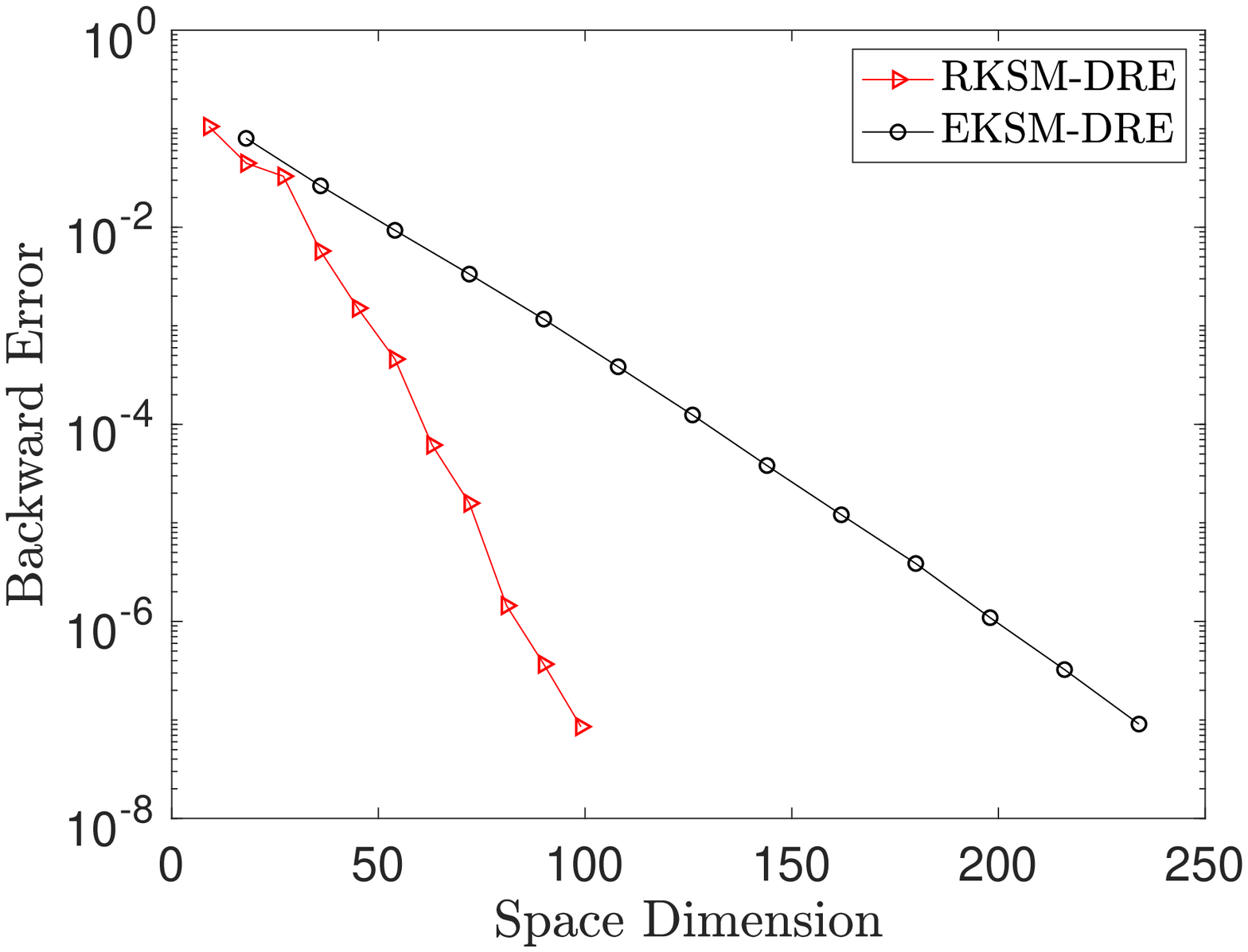}
    \hspace{0.1cm}
    \includegraphics[width=.49\textwidth]{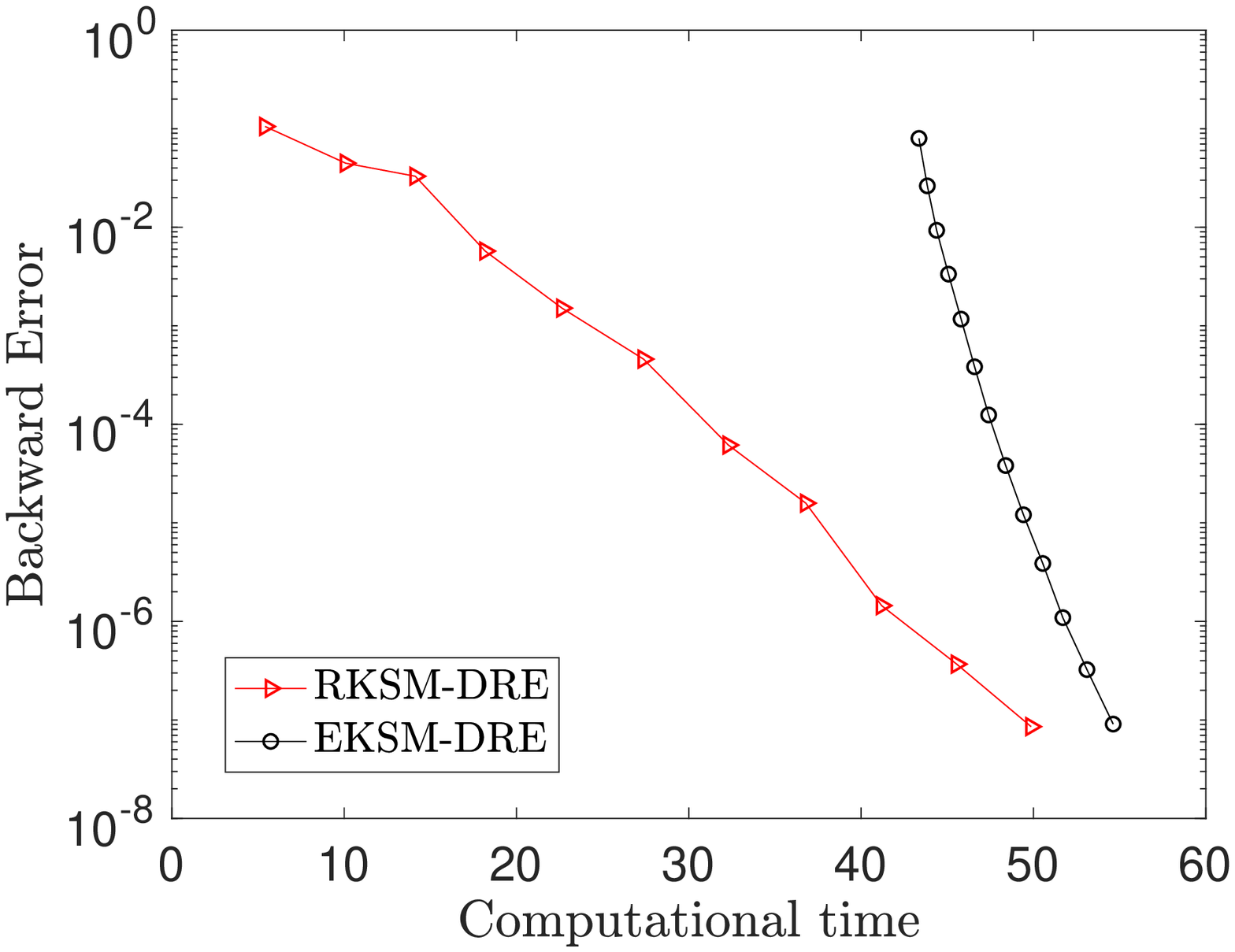}
    \caption[Examples of sparse matrices]{{\sc nsym3d}: Convergence history for 
{\sc eksm-dre} and {\sc rksm-dre}. Left: backward error versus space dimension. Right:
backward error versus computational time.}
    \label{fig:ex2}
\end{figure}

\begin{figure}[htb!]       
    \includegraphics[width=.49\textwidth]{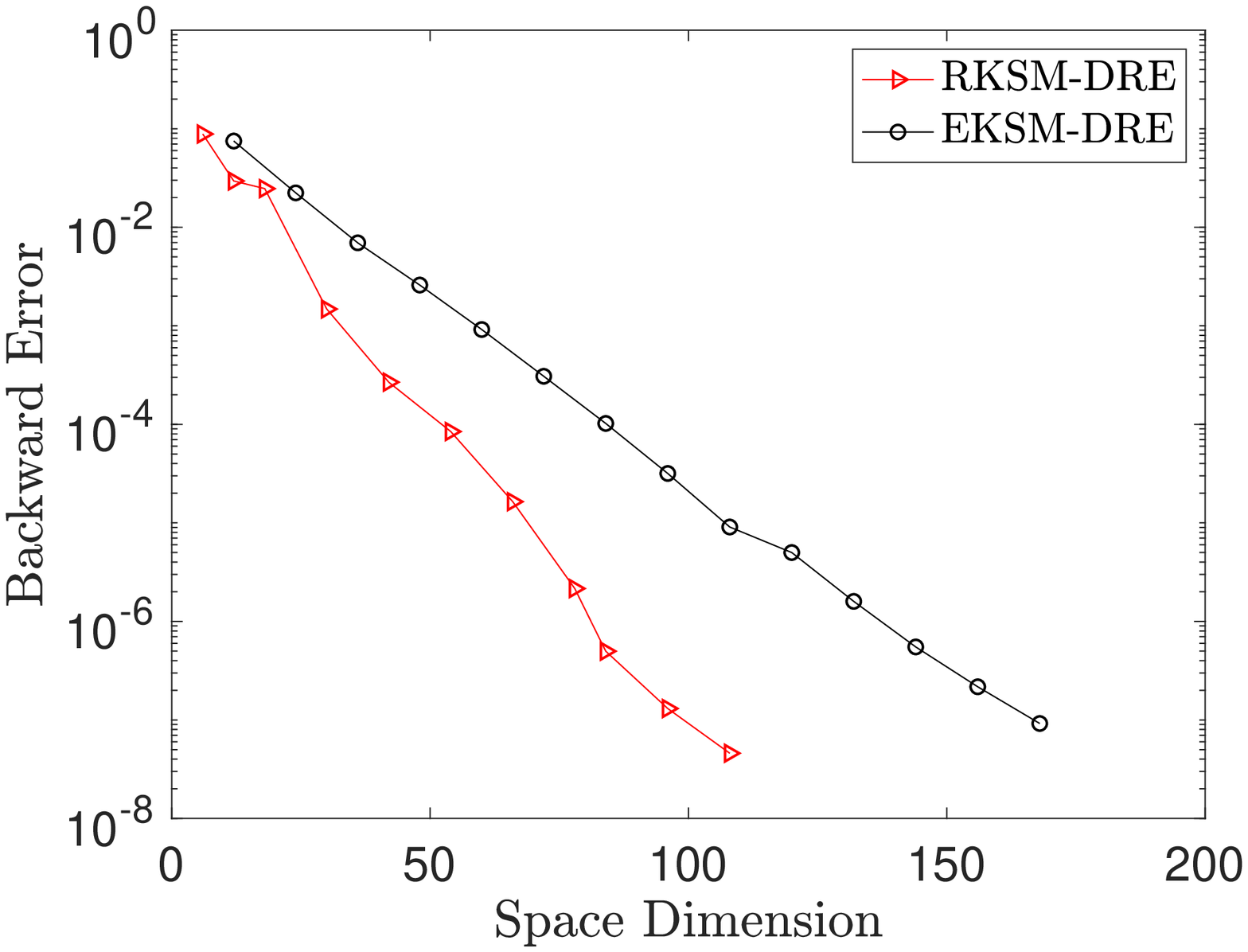}
    \hspace{0.1cm}
    \includegraphics[width=.49\textwidth]{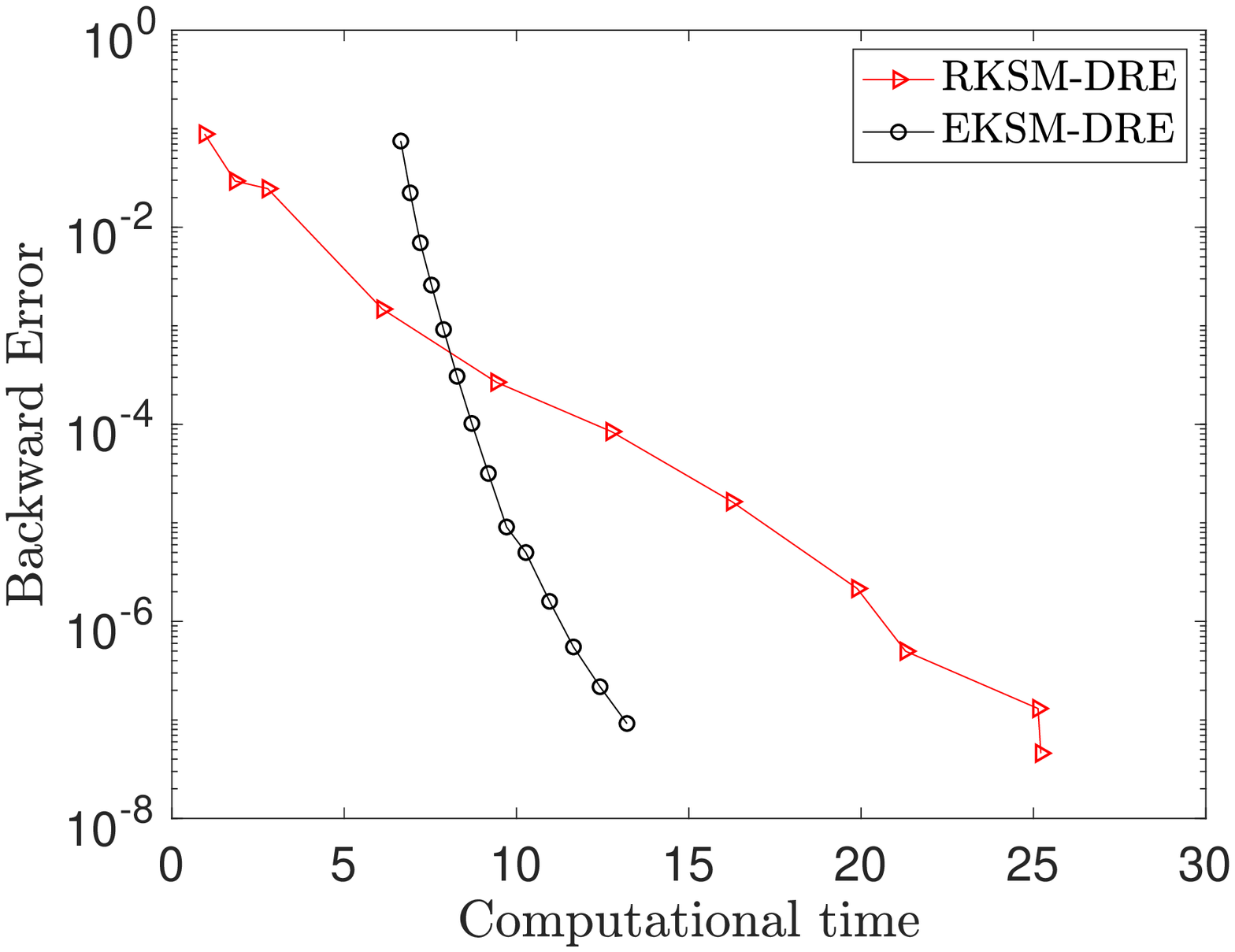}
    \caption[Examples of sparse matrices]{{\sc chip}: Convergence history for
{\sc eksm-dre} and {\sc rksm-dre}. Left: backward error versus space dimension. Right:
backward error versus computational time.}
    \label{fig:ex3}
\end{figure}

\begin{figure}[htb!]       
    \includegraphics[width=.49\textwidth]{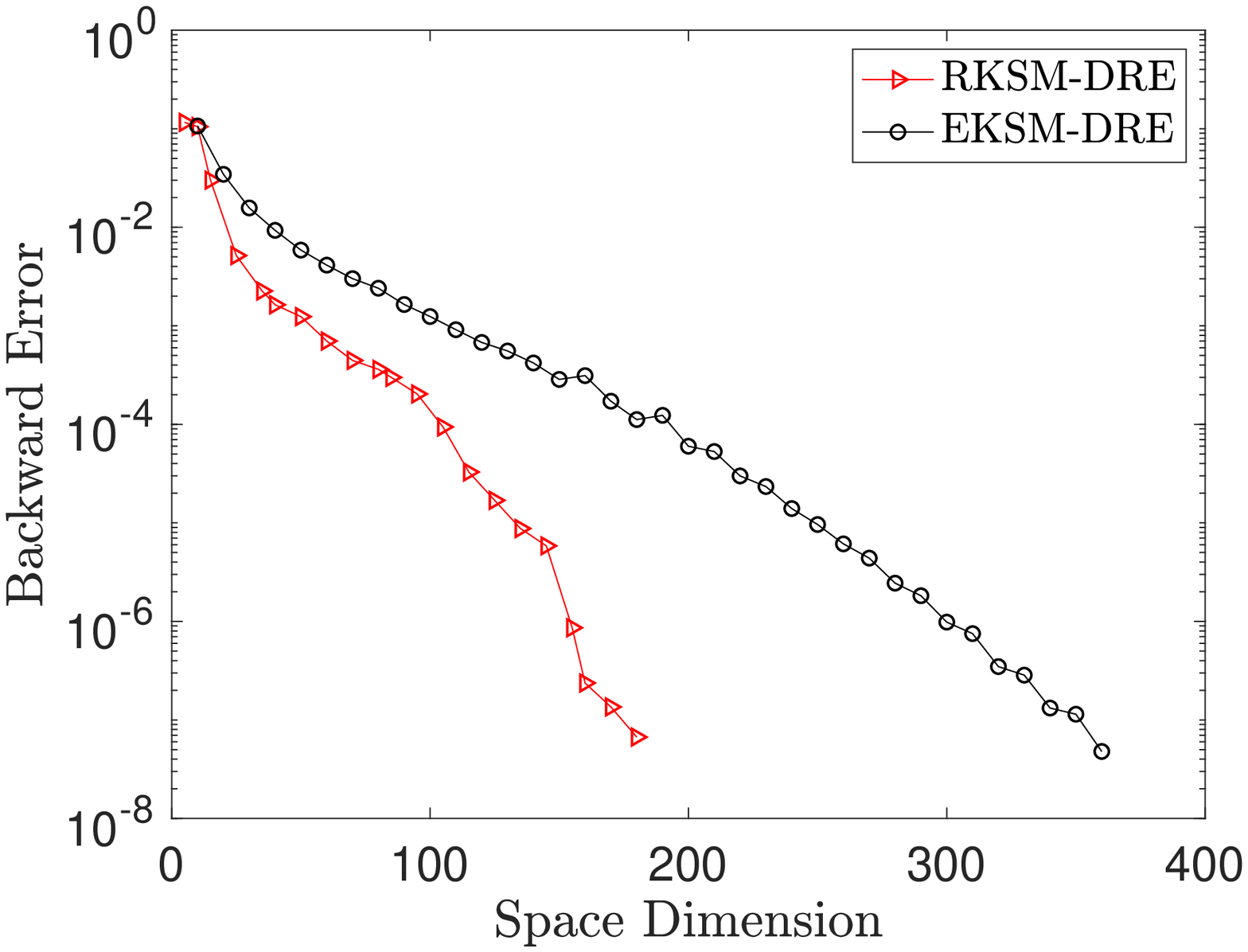}
    \hspace{0.1cm}
    \includegraphics[width=.49\textwidth]{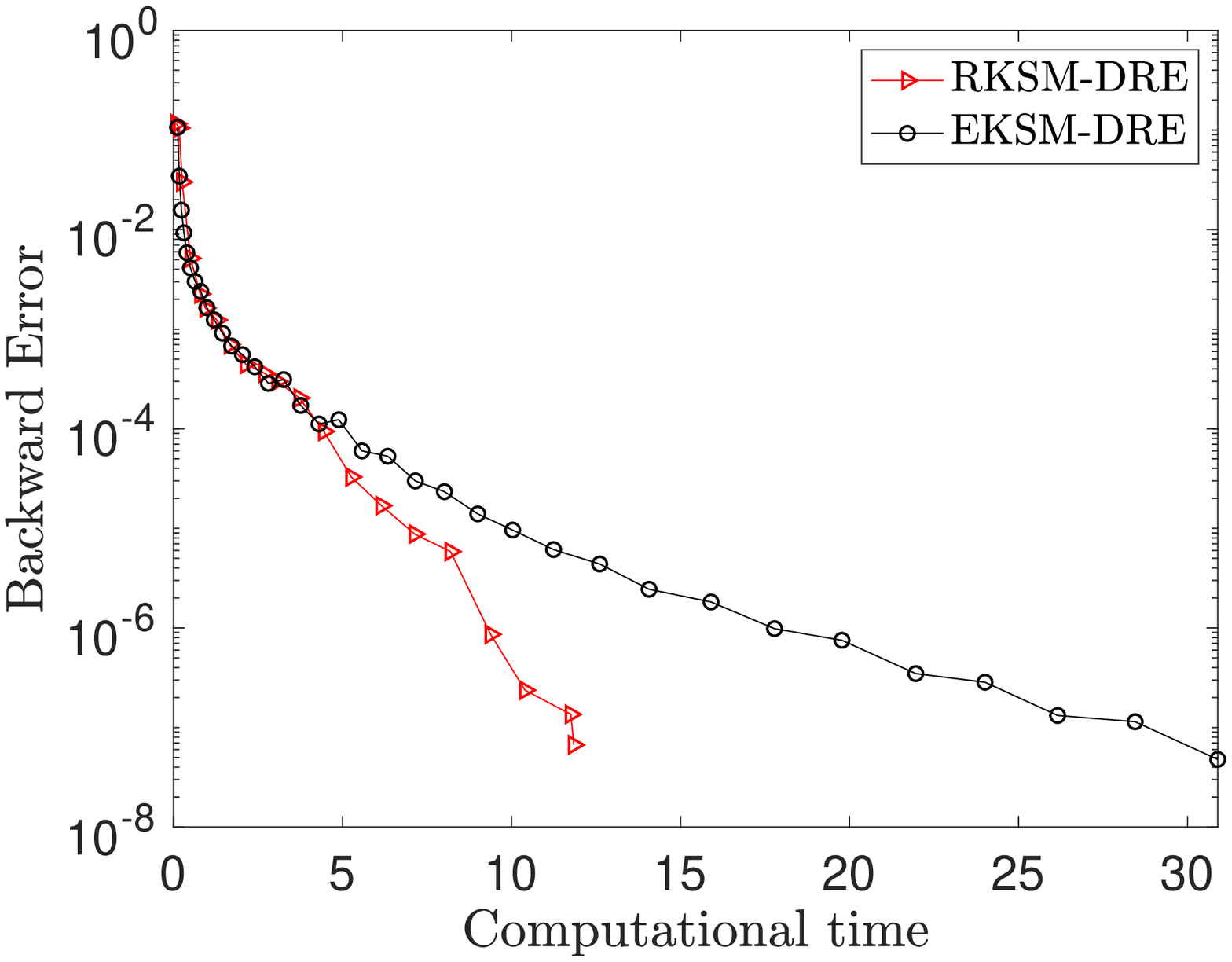}
    \caption[Examples of sparse matrices]{{\sc flow}: Convergence history for
{\sc eksm-dre} and {\sc rksm-dre}. Left: backward error versus space dimension. Right:
backward error versus computational time.}
    \label{fig:ex4}
\end{figure}

\begin{figure}[htb!]       
    \includegraphics[width=.49\textwidth]{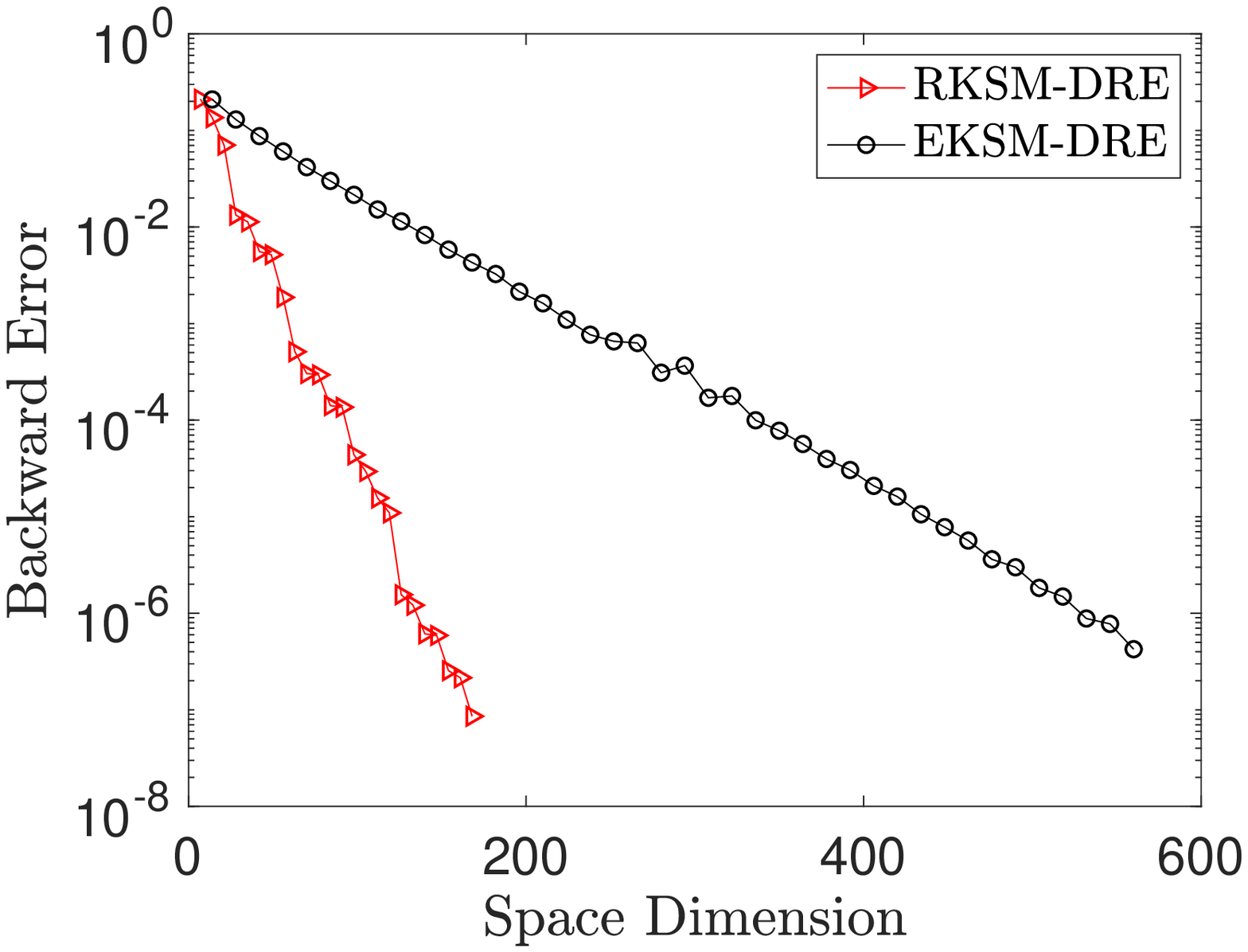}
    \hspace{0.1cm}
    \includegraphics[width=.49\textwidth]{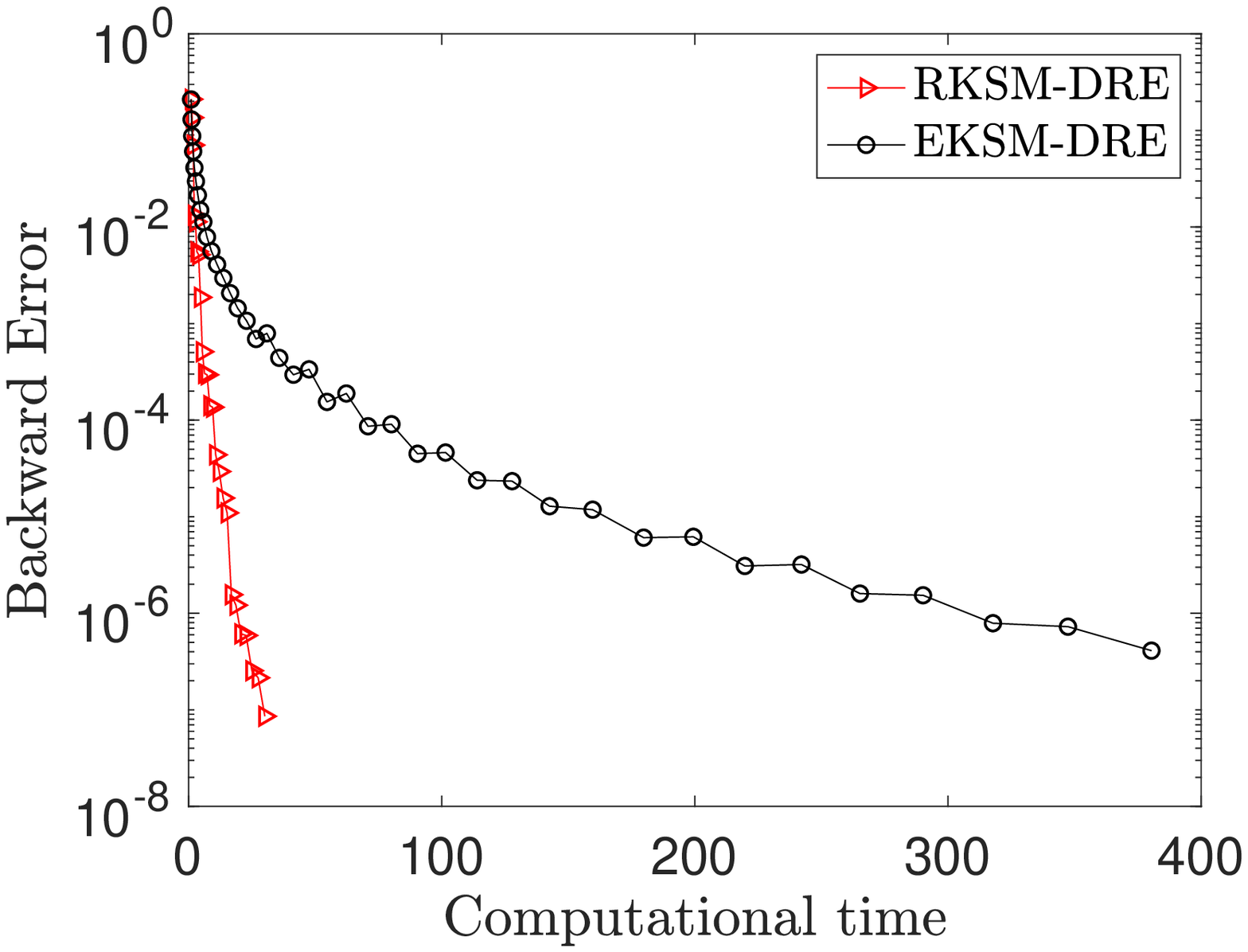}
    \caption[Examples of sparse matrices]{ {\sc rail}: Convergence history for
{\sc eksm-dre} and {\sc rksm-dre}. Left: backward error versus space dimension. Right:
backward error versus computational time.}
    \label{fig:ex5}
\end{figure}
%
For the dataset {\sc sym2d}, the large algebraic linear system  in {\sc rksm-dre} was iteratively solved by
implementing a block conjugate gradient algorithm, with an inner tolerance of $10^{-10}$, preconditioned 
with an incomplete Cholesky factorization with drop tolerance $10^{-4}$. 
For all other datasets, the MATLAB built-in backslash operator was used. 
For {\sc eksm-dre} the coefficient matrix $A$ used to generate the Krylov space
remains constant, hence a 
sparse reordered Cholesky (for {\sc sym2d} and {\sc rail}) or LU (for all other datasets) factorization was 
performed once and for all at the start of the algorithm. Therefore, only sparse triangular solves 
are required at each iteration.  Clearly, the cost of the initial factorization 
depends on the size and density of the coefficient matrix.
These two cost stages are particularly noticeable in the right plots of
 \cref{fig:ex2} and \cref{fig:ex3}, where the {\sc eksm-dre} curve starts towards the right of the plot,
while the rest of the computation throughout the iterations is significantly faster.

In the implementation of {\sc rksm-dre} it is possible to decide a priori whether to use only
real or generically complex shifts.
Our experiments showed that complex shifts were unnecessary for {\sc sym2d} and {\sc nsym3d} and, in fact, 
slowed down convergence when used. 
On the other hand, the use of general complex shifts proved to be crucial for the efficient convergence 
of {\sc rksm-dre} for {\sc chip} and {\sc flow}. For the symmetric data in {\sc rail} 
no complex shifts were used.
We mention in passing that both algorithms are implemented so that the inner solves 
of \cref{projdre} and the residual computations are performed at each iteration; for more demanding
data we would advise a user to perform these computations only periodically to save on computational time.

Comparing performance, we observe that the two algorithms have alternating
leadership in terms of computational time, 
but that {\sc rksm-dre} almost consistently requires half the space dimension of {\sc eksm-dre}. 
This is expected as the space dimension of {\sc eksm-dre} increases with twice the number of 
columns per iteration, in comparison to {\sc rksm-dre}. 
This observation is crucial at the refinement step, where it 
could be considerably more expensive to accurately integrate a DRE of dimension 
$2m(p+q)$ in comparison to a DRE with approximately half the dimension.

{To have a clearer picture of how the various steps influence the performance of the methods,
\cref{breakdowntable} depicts the overall computational time for the 
system solves, the orthogonalization steps and the integration of the 
reduced systems for each algorithm.
For {\sc eksm-dre} the CPU time required for the Cholesky and $LU$ factorizations are included in the 
solving time, but indicated in brackets as well. It is particularly 
interesting to notice the small percentage of time required by  {\sc rksm-dre} in 
comparison to  {\sc eksm-dre} for integrating the reduced system, confirming the comment made
in the previous paragraph.}
\begin{center}
\begin{table}[htb!]
\caption{ A breakdown of 
the computational time for the considered methods for the first two datasets.\label{breakdowntable}}
\centering
\begin{tabular}{|l|r|r|r|r|}
\hline
& & System & Orthogonalisation &Integration \\
Data                        & Method       &  solves (s) & steps (s) & steps (s) \\ 

\hline
\multirow{2}{*}{\sc sym2d}  & \sc rksm-dre & 6.1              & 6.9      & 0.4     \\ 
                            & \sc eksm-dre & 8.6 (2.7)        & 12.1      & 1.3     \\ \hline
\multirow{2}{*}{\sc nsym3d}  & \sc rksm-dre & 38.3              & 0.9      & 0.8     \\ 
                            & \sc eksm-dre & 48.6 (43.5)        & 1.6      & 4.0     \\ \hline

\end{tabular}
\end{table}
\end{center}

\vskip 0.1in
{\it Comparisons with other BDF based methods.} 
We compare the two projection methods {\sc rksm-dre} and {\sc eksm-dre} with 
low-rank methods that have been developed following different strategies. The package {\sc m.e.s.s.} \cite{Saak2016}, for instance, can 
solve Lyapunov and Riccati equations, and perform model reduction of 
systems in state space and structured differential algebraic form, with time-variant
and time-invariant data.
For our purposes,  the solvers in {\sc m.e.s.s.} first discretize the time interval, and
then solve the algebraic Riccati equation resulting from the ODE solver at each
time step. Therefore, the approximation strategy employed at each time iteration 
to solve the algebraic problem is completely independent, and the obtained low-rank numerical
solution needs to be stored separately. More precisely, if $\ell$ timesteps
are performed, the procedure requires solving at least $\ell$ AREs of large
dimensions, delivering the corresponding low-rank approximate solutions. Moreover,
the rank of the constant term in the ARE increases with the time step, due to the
way the ODE solver is structured, further increasing the complexity of the ARE numerical treatment.
 In our experiments with {\sc m.e.s.s.} we only requested the approximate
solution at the final stage. If the whole approximate solution matrix is requested at different time instances,
the memory requirements will grow linearly with that.
%
The overall strategy appears to be memory and computational time consuming, therefore
we considered datasets of reduced size for our comparisons, 
as displayed in \cref{data2}. The considered timespans were left unchanged.

%
%

\begin{table}[htb!]\caption{Data information for comparisons between projection-based methods and {\sc m.e.s.s.}}\label{data2}
\centering
\begin{tabular}{|l|r|r|r|r|r|r|r|}
\hline
Name            & $n$      & $p$/$s$/$q$ & $||A||_{F}$ & $||B||_{F}$ & $||C||_{F}$ & $||Z||_{F}$ & $||E||_{F}$ \\ \hline
\sc sym2d & $40000$ & 5/1/1            & $1.3 \cdot 10^{3}$  & $3.0 \cdot 10^{2}$  & $6.7 \cdot 10^{2}$  & $3.0 \cdot 10^{2}$   & $2 \cdot 10^2$  \\ 
\sc nsym3d       & $8000$ & 6/1/3            & $6.1 \cdot10^{2}$  & $7.7 \cdot 10^{1}$  & $1.9\cdot 10^{2}$  & $8.3 \cdot 10^{1}$  & $ 2.8 \cdot 10^2$  \\
\hline
Name            & $n$      & $p$/$s$/$q$ & $||\widehat{A}||_{F}$ & $||\widehat{B}||_{F}$ & $||\widehat{C}||_{F}$ & $||Z||_{F}$ & $||\widehat{E}||_{F}$ \\ \hline
\sc flow            & $9669$            & 5/1/1            & $4.5 \cdot 10^{6}$  & $2.0 \cdot 10^{4}$  & $1.2 \cdot 10^{3}$  & $0$             & $6.8 \cdot 10^{0}$        \\ 
\sc rail            & $20209$            & 7/6/1            & $4 \cdot 10^{-3}$  & $2.1 \cdot 10^{-7}$  & $6.2 \cdot 10^{0}$  & $1.9 \cdot 10^{2}$        & $2 \cdot 10^{-4}$         \\ \hline
\end{tabular}
\end{table}

Our experimental results are displayed in \crefrange{comp1}{comp5}; we remark that now also
the refinement cost is taken into account in the projection methods. 
 In all tables, the code {\sc bdf}($b,\ell$) refers to the BDF method implemented in the refinement procedure of the reduction methods and in the time discretization procedure of {\sc m.e.s.s}.

\begin{center}
\begin{table}[bht]
\caption{ {\sc sym2d}: Storage and computational time comparison of {\sc rksm-dre}, {\sc eksm-dre} and {\sc m.e.s.s.}. 
Reduction phase performed with {\sc bdf}(1,10), refinement phase with {\sc bdf}(2,100).
In {\sc m.e.s.s.} only the approximate solution at the final time is stored, with no solutions at intermediate time instances returned.
\label{comp1}}
\centering
\begin{tabular}{|l|r|r|r|r|r|}
\hline
                & \# $n$-long & Min/Max   & Reduction  & Refine   & Tot CPU  \\
Method          &  Vecs         & rank  & phase(s)   & phase(s) & time(s) \\ \hline
{\sc rksm-dre}  & 66           & 23/43   & 1.5        & 0.19         & 1.7    \\ 
{\sc eksm-dre}  & 144           & 23/43   & 1.9         & 2.8         & 4.7    \\ 
{\sc m.e.s.s.}-{\sc bdf}(1,10) &  988    & 58/75   &   &          & 319.9    \\
{\sc m.e.s.s.}-{\sc bdf}(2,100)&  1032   & 58/86   &                   &          & 4005.4   \\ 
\hline
\end{tabular}
\end{table}
\end{center}

The tables show the storage requirements in terms of $n$-length vectors, 
the minimum and maximum approximate solution rank (with a truncation 
tolerance $10^{-8}$ for the projection methods) within the set of solutions, 
the CPU time break out of projection and refinement phases
for the two projected methods, and finally the total CPU time. The stopping tolerance for all algebraic methods -- that is the two projection
methods and the Newton--Kleinmann-type method used in {\sc m.e.s.s.} to solve each ARE -- is set to $10^{-7}$. 

 In the {\sc m.e.s.s.} software the user can either select a stopping tolerance (to be used
for all solvers within the Newton--Kleinmann strategy)
or a maximum number of iterations. 
We have experimented with both cases, where the maximum number 
of iterations was detected (a-posteriori) as the 
maximum number of iterations required within {\sc m.e.s.s} to reach the tolerance of $10^{-7}$. It was 
observed that, in the majority of cases, avoiding the residual computation may, in fact, slow down 
the computational procedure. This is due to the possibility of performing several unnecessary 
iterations at some timesteps after the desired accuracy has in fact been reached. 
We therefore only report the results of the more realistic, reliable case where a 
stopping tolerance is selected beforehand. Galerkin acceleration is used to boost the performance of Newton--Kleinmann.

\begin{center}
\begin{table}[htb]
\caption{ {\sc nsym3d}: Storage and computational time comparison of {\sc rksm-dre}, {\sc eksm-dre} and {\sc m.e.s.s.}. 
Reduction phase performed with {\sc bdf}(1,10), refinement phase with {\sc bdf}(2,100).
In {\sc m.e.s.s.} only the approximate solution at the final time is stored, with no solutions at intermediate time instances returned.
\label{comp2}}
\centering
\begin{tabular}{|l|r|r|r|r|r|}
\hline
                & \# $n$-long & Min/Max & Reduction  & Refine   & Tot CPU  \\
Method          &  Vecs         & rank & phase(s)   & phase(s) & time(s) \\ \hline
{\sc rksm-dre}  & 108         & 36/66   & 4.5       & 4.6        & 9.1    \\ 
{\sc eksm-dre}  & 196          & 36/66   & 2.8       & 5.7        & 8.5    \\
{\sc m.e.s.s.}-{\sc bdf}(1,10) &   1116  & 71/90   &   &          & 431.0    \\
{\sc m.e.s.s.}-{\sc bdf}(2,100)&  1152   & 67/94   &   &          & 4965.0    \\
\hline
\end{tabular}
\end{table}
\end{center}

All numbers in the tables illustrate the large computational costs of {\sc m.e.s.s.},
as expected by the strategy ``first time-discretize, then solve'',
whereas both projection
methods require just a few seconds of CPU in most cases.

The storage requirements of both reduction methods is independent of the number of timesteps where the solution is required.}
This is due to the fact that only a few $n$-long basis vectors need to be generated and stored,
while only the reduced problem solution $Y_m(t)$ changes at the timesteps $t$.
The memory requirements of {\sc m.e.s.s.} are measured as the dimensions of the low-rank factor returned by the Newton-Kleinmann procedure, before column compression, at the final timestep. The dimension decreases significantly with the column compression. In our experiments we only stored the approximate
solution at the last time step, however memory will be correspondingly higher if 
the whole approximation matrix is required at more instances (memory will thus grow linearly with the number
of time instances to be monitored).

Between the two projection methods, we observe that the extended space yields a significantly 
larger basis than the actual
approximate solution rank it produces. This means that the approximate solution belongs to a much smaller
space than the one constructed by {\sc eksm-dre}. This is far less so with {\sc rksm-dre}. The different 
behavior confirms what has been already observed for the two methods in the ARE case \cite{Simoncini2014}.

\begin{center}
\begin{table}[htb]
\caption{ {\sc flow}: Storage and computational time comparison of {\sc rksm-dre}, {\sc eksm-dre} and {\sc m.e.s.s.}. 
Reduction phase performed with {\sc bdf}(1,10), refinement phase with {\sc bdf}(2,100).
In {\sc m.e.s.s.} only the approximate solution at the final time is stored, with no solutions at intermediate time instances returned.
\label{comp3new}}
\centering
\begin{tabular}{|l|r|r|r|r|r|}
\hline
                & \# $n$-long & Min/Max & Reduction  & Refine   & Tot CPU  \\
Method          &  Vecs         & rank & phase(s)   & phase(s) & time(s) \\ \hline
{\sc rksm-dre}  & 186            & 95/100   & 12       & 4.6        & 16.6    \\ 
{\sc eksm-dre}  & 372          & 95/100   & 30.8         & 23.7         & 54.5    \\
{\sc m.e.s.s.}-{\sc bdf}(1,10) &  1280   & 87/106   &   &          &431.7   \\
\hline
\end{tabular}
\end{table}
\end{center}

\begin{center}
\begin{table}[htb]
\caption{ {\sc rail}: Storage and computational time comparison of {\sc rksm-dre}, {\sc eksm-dre} and {\sc m.e.s.s.}. 
Reduction phase performed with {\sc bdf}(1,10), refinement phase with {\sc bdf}(2,100).
In {\sc m.e.s.s.} only the approximate solution at the final time is stored, with no solutions at intermediate time instances returned.
\label{comp5}}
\centering
\begin{tabular}{|l|r|r|r|r|r|}
\hline
                & \# $n$-long & Min/Max & Reduction  & Refine   & Tot CPU  \\
Method          &  Vecs         & rank & phase(s)   & phase(s) & time(s) \\ \hline
{\sc rksm-dre}  & 168         & 153/160   & 6.4      &   3.3      & 9.7    \\ 
{\sc eksm-dre}  & 462          & 153/160   & 39.2       & 5.7        & 44.9    \\
{\sc m.e.s.s.}-{\sc bdf}(1,10) & 6345   & 151/158   &   &          & 705.3   \\
{\sc m.e.s.s.}-{\sc bdf}(2,100)&    4023 & 124/158   &   &          & 3396.5   \\
\hline
\end{tabular}
\end{table}
\end{center}


\vskip 0.1in
{\it Comparisons with splitting methods.}
{We next compare {\sc rksm-dre} with the fourth order additive splitting method 
({\sc split-add4($\ell$)}) developed in \cite{Stillfjord2018}. The method is based on 
splitting the DRE into the linear and non-linear subproblems, for which respective closed 
form solutions exist and are explicitly approximated. The numerical solutions to the 
subproblems are then recombined to approximate the solution to the full problem, by means 
of an additive splitting scheme. The main computational effort is due to the repeated evaluation 
of matrix exponentials, which has been resolved by using a Krylov-based matrix exponential approximation. 
Similar to the issue discussed with {\sc m.e.s.s}. in the previous section, the $\ell$ (factored) solution 
matrices are independently calculated at each timestep, leading to significant memory requirements.

 To ensure that we are comparing methods with similar approximation accuracies, we generate 
reference solutions $X_{ref}(t_j)$ for the selected time instances $t_j$. 
This is done by using {\sc rksm-dre} with a
stopping tolerance of $10^{-10}$, plus a refinement process with {\sc bdf}($4,10^{4}$) from \cite{Saak2016}.
To allow for such accurate approximations, we consider slightly smaller problem dimensions for
the first two datasets, and we set $p=s=1$ and $X_0=0$.
 

{The input parameters are tailored so that the approximate solutions from different methods have relatable accuracies. 
In particular, {\sc rksm-dre} is solved with an outer stopping tolerance of 
$10^{-6}$ and with {\sc bdf}(3,1000) in the refinement process. The number of 
timesteps utilized in {\sc split-add4} is selected as $\ell = 500$. 
The expected approximation errors  relative to the reference solution, measured as
$$ 
\frac{\|X_{approx}(t) - X_{ref}(t)\|_{F} }{\|X_{ref}(t)\|_{F}},
$$
 are illustrated in \cref{fig:approxerr} (dataset {\sc sym2d} in the left plot, dataset {\sc nsym3d} in the right plot).
 \begin{figure}[htb!]       
    \includegraphics[width=.49\textwidth]{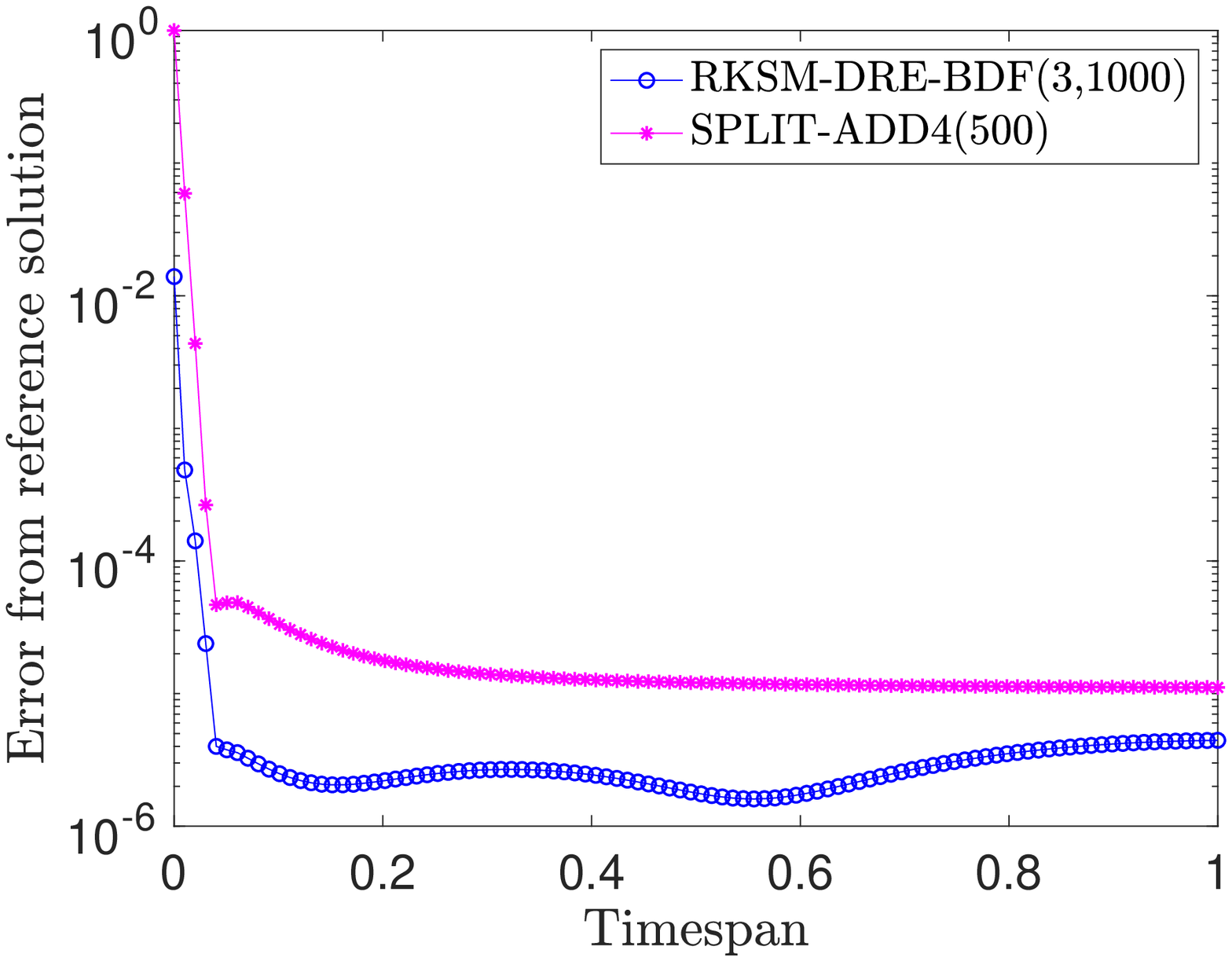}
    \hspace{0.1cm}
    \includegraphics[width=.49\textwidth]{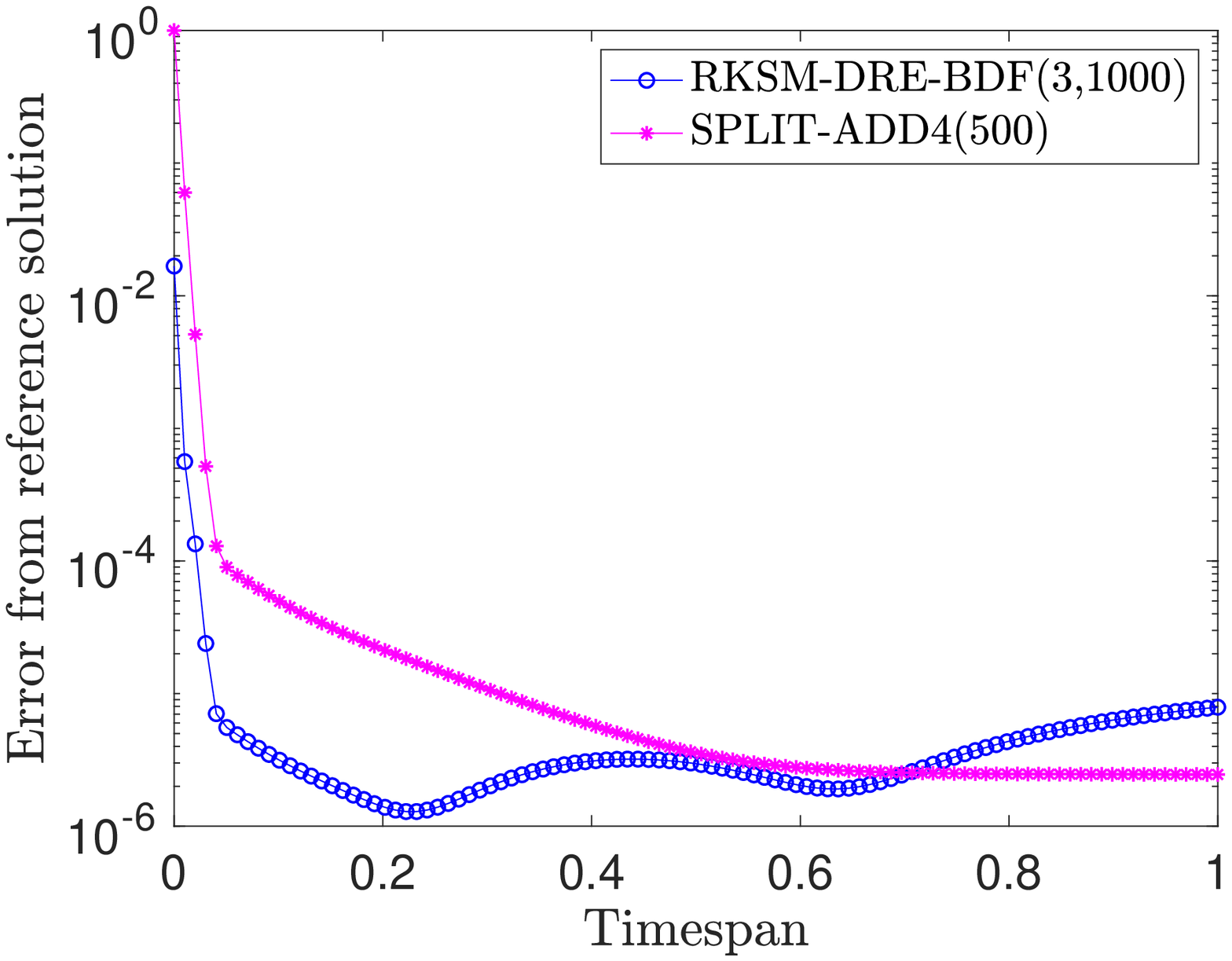}
    \caption[Examples of sparse matrices]{ Expected approximation error for {\sc rksm-dre} and 
{\sc split-add4(500)}. Left: Dataset {\sc sym2d}. Right: Sataset {\sc nsym3d}.}
    \label{fig:approxerr}
\end{figure}
The figures indicate that we compare methods having approximation errors of 
similar order. The performance results are contained in \cref{splittable} for two 
different discretizations of {\sc sym2d} and {\sc nsym3d}.}
\begin{center}
\begin{table}[htb!]
\caption{Storage and computational time comparison of {\sc rksm-dre} and {\sc split-add4(500)}. Reduction 
phase performed with {\sc bdf}(1,10), refinement phase with {\sc bdf}(3,1000). In {\sc split-add4}
only the approximate solution at the final time is stored.} \label{splittable}
\centering
\begin{tabular}{|l|r|r|r|r|}
\hline
& &\# $n$-long  & Min/Max &Tot CPU \\
Data ($n$)                       & Method       &  Vecs & rank & time (s) \\ 

\hline
\multirow{2}{*}{\sc sym2d $(10^{4})$}  & \sc rksm-dre & 8             & 3/6     & 0.6     \\ 
                            & {\sc split-add4(500)} &  28        &3/7      & 34.9     \\ \hline
                        
\multirow{2}{*}{\sc nsym3d $(8\cdot 10^{3})$}  & \sc rksm-dre & 10           & 4/7      & 2.2     \\ 
                            &{\sc split-add4(500)}& 36     & 3/9     & 37.9    \\ \hline
\multirow{2}{*}{\sc sym2d $(9 \cdot 10^{4})$}  & \sc rksm-dre & 6              & 3/4    & 1.2     \\ 
                            &{\sc split-add4(500)} &  28      & 3/7      & 330.0     \\ \hline
                         
\multirow{2}{*}{\sc nsym3d $(2.7 \cdot 10^4)$}  & \sc rksm-dre & 10           & 4/8     & 10.1     \\ 
                            & {\sc split-add4(500)} & 36      & 3/9     & 127.8    \\ \hline

\end{tabular}
\end{table}
\end{center}
All numbers indicate the competitiveness of {\sc rksm-dre} in terms of storage and computational time. The memory requirements for {\sc split-add4} is measured as the dimension of the solution factor at the final timestep, before column compression. If the solution is required at more time instances, then these memory requirements will increase accordingly.

We also mention that we have experimented with the dynamic splitting methods introduced in \cite{Mena2018},
however the algorithms proposed by the authors\footnote{We thank Chiara Piazzola for providing us with
her Matlab implementation of the method.} in \cite{Mena2018} appeared to be better suited for small to medium size problems.
\vskip 0.1in
{\it Discussion on the refinement step.}
In previous sections, we have stressed that the two approximation stages of the projection method are
independent, and we have focused on determining an effective approximation space. Here we
linger over the accuracy of the second stage, the refinement step. Exploiting the far smaller
problem size of the reduced problem, it is possible to allow for a much more accurate integration phase
than what was done during the iteration of the reduction step. This crucial fact is already illustrated
in the time break down of \crefrange{comp1}{comp3new}, where especially for {\sc rksm-dre} the refinement
phase employs a fraction of the overall computational time, while still allowing for a rather
accurate final solution.

We next explore in more detail these advantages with {\sc rksm-dre}
on {\sc sym2d}, where the discretization was further refined to get a coefficient matrix of dimension $10^6$.
The dimensions of the other corresponding matrices remain as presented in \cref{data}.  We investigate the time taken by DRE solvers with different accuracies to emphasize the advantages and 
flexibility of the refinement procedure. \cref{tensix} reports the timings for a 
refinement step performed by three different {\sc bdf} methods and
three splitting methods. The 8th order adaptive splitting method ({\sc split-adapt}8) also  comes from \cite{Stillfjord2018} and is performed with a tolerance of $10^{-7}$, i.e., the same as that of the reduction procedure. 
We emphasize that in the refinement phase we have utilized some of the most accurate integrators available,
and nevertheless the high-dimensional ($n = 10^6$) problem is approximated in 
less than 200 seconds for all integrators. 


\begin{center}
\begin{table}
\centering
\caption{{\sc sym2d} (of size $10^{6}$): Results with {\sc rksm-dre}, 
using different refinement strategies. Reduction phase performed with {\sc bdf}(1,10) and tolerance $10^{-7}$.}\label{tensix}
\begin{tabular}{|l|r|r|r|r|r|}
\hline
Refinement &  \# $n$-long & Soln.& Reduction & Refinement  & Tot CPU \\ 
Method & Vecs & rank & phase(s) & phase(s) & time(s) \\
\hline
\sc bdf(2,100)       & 72       & 55   & 37.7     & 1.1         & 38.8     \\
\sc bdf(3,1000)       & 72       & 55   & 37.7     & 9.6          & 47.3    \\
\sc bdf(4,10000)       & 72       & 55   & 37.7     & 95.5          & 133.2    \\
\sc split-add4(500)     & 72       & 55   & 37.7     & 25.9        & 63.6   \\
\sc split-add8(500)     & 72       & 55   & 37.7     & 59.6         & 97.3   \\ 
\sc split-adapt8     & 72       & 55   & 37.7     & 160.7         &198.4   \\ \hline
\end{tabular}
\end{table}
\end{center}

\section{Conclusions and open problems} \label{sec:conclusion}
We have devised a rational Krylov subspace based order reduction method for solving the symmetric differential Riccati equation,
providing a low-rank approximate solution matrix at selected time steps. 
A single projection space is generated for all time instances, and the space is 
expanded until the solution is sufficiently accurate. We stress that our approach is very general, and that it could be
applied to subspaces other than Krylov-based ones, as 
long as the spaces are nested, so that they keep growing as
the iterations proceed. This methodology could then be employed for more complex
settings, such as parameter dependent problems, where the involved approximation
space may require the inclusion of some parameter sampling.

Like in typical model order reduction strategies, in our methodology
time stepping is only performed at the reduced level, so that
the integration cost is drastically lower than what one would have by applying the time stepping
on the original large dimensional problem.
We have derived a new stopping criterion that takes into account the different approximation behavior
of the algebraic and differential portions of the problem, together with a refinement procedure
that is able to improve the final approximate solution by using a high-order integrator.
These enhancement strategies have also been applied to the extended Krylov subspace approach.
{We have analyzed the asymptotic behavior of the reduced order solution, so as to ensure that
the generated approximation behaves like the sought after time-dependent solution.

Although our numerical results are promising, there are still several open issues
associated with the reduced order solution of the DRE.
In particular, while stability and other matrix properties associated with the
solutions $X(t)$ have been thoroughly studied 
\cite{Bitmead1985,Gevers1986,Poubelle1988,Dieci.Eirola.96}, 
the analysis of corresponding properties for the approximate solution $X_m(t)={\cal V}_m Y_m(t) {\cal V}_m^T$ 
for $t \in [0, t_f]$ is still a largely open problem. In \cite{Koskela2018} some
interesting monotonicity properties have been shown when the polynomial Krylov subspace
is used together with particular ODE solvers;
a complete analysis for $X_m(t)$ in a more general setting would be desirable.


\section*{Acknowledgements} We thank Khalide Jbilou for helpful explanations on \cite{Angelova2018}, and
Jens Saak for assistance in using the software M.E.S.S. v.1.0.1 \cite{Saak2016}. We are also grateful
to Chiara Piazzola for making her code from 
\cite{Mena2018} available and to Tony Stillfjord for pointing us towards his codes from \cite{Stillfjord2018}. 
We thank Maximilian Behr for some helpful criticism on a previous version of this manuscript.
The first author acknowledges the insightful comments and kind hospitality of the group of
Peter Benner at the Max Planck Institute in Magdeburg (D), during a short visit in February 2019. Furthermore, we thank two anonymous referees for their careful reading and insightful comments.

This research is supported in part by the Alma Idea grant of
the Alma Mater Studiorum Universit\`a di Bologna, and by INdAM-GNCS under the 2018 Project 
``Metodi numerici per equazioni lineari, non lineari e matriciali con applicazioni''.
The second author is a member of INdAM-GNCS.

\appendix
\section{Krylov subspace properties} \label{matral}
In this Appendix we review some properties of extended and rational Krylov 
subspaces. As in \cref{sec:projection} we denote $N=[C^T,Z]$.
\vskip 0.1in 
{\it Extended Krylov subspace.}  The extended Krylov subspace $\mathcal{EK}_{m}(A^{T},N)$ takes the form discussed in \cref{sec:projection}.
The orthonormal basis $\mathcal{V}_{m} \in \mathbb{R}^{n \times 2m(p+q)}$ spanning the subspace is formed using the extended Arnoldi algorithm \cite{Druskin1998}. 
Let 
\begin{equation}
\widetilde{\mathcal{T}}_{m}^{T} = \mathcal{V}_{m+1}^{T}A^{T}\mathcal{V}_{m} =
\begin{bmatrix}
\mathcal{T}_{m}^{T}\\
t_{m+1,m}E_{2m}^{T}
\end{bmatrix} \in \mathbb{R}^{2(m+1)(p+q) \times 2m(p+q)},
\end{equation}
where $\mathcal{V}_{m+1}  = [\mathcal{V}_{m} \hspace{0.1cm} V_{m+1}] \in \mathbb{R}^{n \times 2(m+1)(p+q)}$ and $E_{2m}$ is the last $2(p+q)$ columns of  $I_{2m(p+q)}$. 
The extended Arnoldi algorithm produces the Arnoldi-type relation
\begin{equation}
\begin{split}
\label{extrel}
A^{T}\mathcal{V}_{m} &= \mathcal{V}_{m+1}\widetilde{\mathcal{T}}_{m}^{T} 
 = \mathcal{V}_{m}\mathcal{T}_m^{T} +  V_{m+1}t_{m+1,m}{E}_{2m}^{T}.
\end{split}
\end{equation}

\vskip 0.1in
{\it Rational Krylov subspace.}
The rational Krylov subspace was originally proposed in the eigenvalue context in \cite{Ruhe1984}.
Its use in our context is motivated by \cite{Simoncini2014} and later \cite{Simoncini2016}, where 
its effectiveness in the solution of the algebraic Riccati equation is amply discussed.

Assume that $A$ is Hurwitz.
 Given ${\pmb s} = \{s_{1}, s_{2}, \dots\}$, with $s_{j} \in \mathbb{C}^{+}$, the rational Krylov subspace is given by
${\cal RK}_{m}(A,N,{\pmb s})$ as defined in \cref{sec:projection}.
The approximation effectiveness of this subspace depends on the choice of shifts ${\pmb s}$, and this
issue has been investigated in the literature; see, e.g., \cite{Penzl2000}, \cite{Druskin2011}.
The adaptive choice of shifts was tailored to the ARE in \cite{Lin2015} by the inclusion of 
information of the term $BB^{T}$ during the shift selection; see also \cite{Simoncini2016} for
a more detailed discussion\footnote{The Matlab code of the rational Krylov subspace method for ARE is available
at {\tt http://www.dm.unibo.it/$\,\tilde {\,}$simoncin/software}}. In our numerical experiments we used this 
last adaptive strategy, where the approximate solution at timestep $t_f$ is used.

The algorithm presented in \cite{Druskin2011} forms a complex basis, when the shifts are not all 
real.  In short, when $s_{j} \in \mathbb{C}^{+}$, the original approach would be to use the shift $s_{j}$ to form the next block $V_{j}$ and to then let the following shift be given by $s_{j+1} = \overline{s}_{j}$, where $\overline{s}_{j}$ denotes the complex conjugate of $s_{j}$. This results in both $V_{j}$ and $V_{j+1}$ being complex. 
As a consequence, the reduced DRE has complex coefficient matrices, although
the final resulting approximations $X_{m}(t)$ will be real. Standard ODE solvers do not handle 
complex arithmetic well, hence we implemented an all-real basis using the method 
introduced in \cite{Ruhe1994}, which works as follows.
 If the shift $s_{j}$ is complex then the block $W_{j} = (A - s_{j}I)^{-1}V_{j-1}$ 
is also complex, hence
we split it 
into its real and complex parts, that is $W_{j} = W_{j}^{(r)} + W_{j}^{(c)} \imath$.
The block $V_j$ is then formed by orthogonalizing $W_{j}^{(r)}$ with respect to all
vectors in the already computed basis,  
 after which $V_{j+1}$ is formed by orthogonalizing $W_{j}^{(c)}$ with respect to all previous
vectors in in the computed basis, and in $V_{j}$. 
This determines the same space, since span$\{W_j, \bar W_j\} = $span$\{V_j, V_{j+1}\}$.
%
The resulting {\it real} basis of the rational Krylov subspace is given by
 $\mathcal{V}_{m} = [V_{1}, \dots, V_{m}] \in \mathbb{R}^{n \times m(p+q)}$. 
We also define the matrices  $\mathcal{V}_{m+1} = [ \mathcal{V}_{m}, V_{m+1} ] \in \mathbb{R}^{n \times (m+1)(p+q)}$ and the matrix 
\begin{equation}
\widetilde{\mathcal{H}}_{m} = 
\begin{bmatrix}
\mathcal{H}_{m}\\
r_{m+1,m}E_{m}^{T}
\end{bmatrix} \in \mathbb{R}^{(m+1)(p+q) \times m(p+q)},
\end{equation}
where $r_{m+1,m} \in \mathbb{R}^{(p+q) \times (p+q)}$ and $E_{m}$ holds the last $(p+q)$ columns of $I_{m(p+q)}$. 
The matrix $\widetilde{\mathcal{H}}_{m}$ contains the orthogonalization coefficients obtained during the rational Arnoldi algorithm.

Let $\mathcal{T}_{m}^{T} = \mathcal{V}_{m}^{T}A^{T}\mathcal{V}_{m} \in \mathbb{R}^{m(p+q) \times m(p+q)}$. 
The rational Krylov basis satisfies the Arnoldi-type relation
\begin{equation}
\label{ratrel}
A^{T}\mathcal{V}_{m} = \mathcal{V}_{m}\mathcal{T}_{m}^{T} + \widehat{V}_{m+1}G_{m}^{T},
\end{equation}
where $G_{m}^{T} = \mathbf{\gamma}r_{m+1,m}E_{m}^{T}\mathcal{H}_{m}^{-1}$ 
and the matrix $\widehat{V}_{m+1}$ is an orthonormal matrix such that 
\begin{equation}
\widehat{V}_{m+1}\mathbf{ \gamma } = V_{m+1}s_{m} - (I_{n} - \mathcal{V}_{m}\mathcal{V}_{m}^{T})A^{T}V_{m+1}
\end{equation}
is the $QR$ decomposition of the matrix on the right (see \cite{Druskin2011, Lin2013}).
The rational Krylov procedure requires as an extra input the (usually real) values $s_{0}^{(1)}, s_{0}^{(2)}$, 
which form a rough approximation of a spectral region used to compute the next shift. 
The reader is referred to \cite{Druskin2011, Simoncini2016} for implementation details. 
Further, for the computation of the term $G_{m}^{T}Y_{m}(t)$ contained in the residual computation of {\sc rksm-dre}, 
we follow an accelerated computation technique presented in \cite{Druskin2011}.

\section{Extended Krylov subspace based method}
\label{appb}
The extended Krylov ({\sc eksm-dre}) subspace method for solving \cref{dre} is presented in \cref{alg:eksmdre}. 
\begin{algorithm}
\caption{EKSM-DRE}
\label{alg:eksmdre}
\begin{algorithmic}[0]
\REQUIRE{$A \in \mathbb{R}^{n \times n}$, $B \in \mathbb{R}^{n \times s}$, $C \in \mathbb{R}^{p \times n}$, $Z \in \mathbb{R}^{n \times q}$, $tol$, $t_{f}$, $\ell$}
\STATE{(i) Perform reduced $QR$: $\left([C^{T},\, Z], \, A^{-1}[C^{T},\, Z] \right) = V_{1}\Lambda_{1}$}
\STATE{\hskip 0.2in Set $\mathcal{V}_{1} \equiv V_{1}$}
	\STATE{\hskip 0.2in {\bf for}  $ m = 2,3 \dots$}
	\STATE{\hskip 0.4in Compute the next basis block $V_{m}$} 
	\STATE{\hskip 0.4in Set $\mathcal{V}_{m} = [\mathcal{V}_{m-1}, V_{m}]$}
	\STATE{\hskip 0.4in Update $\mathcal{T}_{m}$ as in \cite{Simoncini2007} and 
$B_{m} 	= \mathcal{V}_{m}^{T}B$, $Z_{m} = \mathcal{V}_{m}^{T}Z$  and $C_{m} = C\mathcal{V}	_{m}$}
	\STATE{\hskip 0.4in Integrate \cref{projdre} from 0 to $t_f$ using BDF($1,\ell$)}
	\STATE{\hskip 0.4in Compute $\rho_m$ using \cref{cheapy} where $\tau_{m}^{T} = t_{m+1,m}E_{2m}^{T}$}	 
		\STATE{\hskip 0.4in {\bf if} $\rho_m < tol$}
		\STATE{\hskip 0.6in \textbf{go to} (ii)}
		\STATE{\hskip 0.4in {\bf end if}}
\STATE{\hskip 0.2in{\bf end for}}
\STATE{(ii) Refinement: solve \cref{projdre} with a more accurate integrator}
\STATE{Compute $Y_{m}(t_j) = \widehat{Y}_{m}(t_j)\widehat{Y}_{m}(t_j)^{T}$, $j=1, \ldots, \ell$  using the truncated SVD}
\RETURN $\mathcal{V}_{m} \in \mathbb{R}^{n \times 2m(p+q)}$ and 
$\ell$ factors $\widehat{Y}_{m}(t_j) \in \mathbb{R}^{2m(p+q) \times r}$, $j=1, \ldots, \ell$.
\end{algorithmic}
\end{algorithm}

\bibliographystyle{siam}
\bibliography{Drereferencesreal}

\end{document}